\documentclass{amsart}

\usepackage{a4wide, amssymb, mathrsfs, stmaryrd}
\usepackage{fancyheadings}
\usepackage{amscd}

\usepackage[
open,
openlevel=2,
atend
]{bookmark}[2009/08/13]

\usepackage{amsmath}

\usepackage[T1]{fontenc}

\usepackage{xypic}

\usepackage[cmtip,all,2cell]{xy}

\usepackage{color}

\UseAllTwocells

\def\id{\textrm{id}}

\def\xto#1{\xrightarrow{#1}}

\newcommand{\nc}{\newcommand}

\newcommand{\myMO}[1]{{\fontshape{rm}{\textbf{#1}}}}

\DeclareMathOperator{\Grmod}{\myMO{Grmod}\hspace{+0.25ex}-\hspace{-0.25ex}}
\DeclareMathOperator{\Mod}{\myMO{Mod}\hspace{+0.25ex}-\hspace{-0.25ex}}
\DeclareMathOperator{\proj}{\myMO{proj.dim}}
\DeclareMathOperator{\gl}{\myMO{gl.dim}}
\DeclareMathOperator{\dimm}{\myMO{dim}}

\newcommand*\cocolon{%
        \nobreak
        \mskip6mu plus1mu
        \mathpunct{}%
        \nonscript
        \mkern-\thinmuskip
        {:}%
        \mskip2mu
        \relax
}

\nc{\Ch}{\operatorname{Ch}}
\nc{\sA}{{\mathscr A}}
\nc{\wt}{\widetilde} \nc{\bl}{\bullet} \nc{\al}{\alpha}
\nc{\sg}{\sigma} \nc{\vf}{\varphi} \nc{\om}{\omega}
\nc{\ve}{\varepsilon} \nc{\ol}{\overline} \nc{\lb}{\lambda}
\nc{\Lb}{\Lambda} \nc{\Gm}{\Gamma} \nc{\cP}{{\mathscr P}}
\nc{\sB}{{\mathscr B}}
\nc{\ul}{\underline} \nc{\os}{\overset} \nc{\us}{\underset}
\nc{\pa}{\partial} \nc{\wh}{\widehat} \nc{\sbs}{\subset} \nc{\br}{\breve}
\nc{\lra}{\longrightarrow} \nc{\all}{\allowdisplaybreaks}
\nc{\Ker}{\operatorname{Ker}} \nc{\Img}{\operatorname{Im}}
\nc{\Kan}{\operatorname{Kan}} \nc{\Hom}{\operatorname{Hom}}
\nc{\Imm}{\operatorname{Im}}   \nc{\Ho}{\operatorname{Ho}}
\nc{\Ext}{\operatorname{Ext}}    \nc{\Cone}{\operatorname{Cone}}
\nc{\pr}{\operatorname{pr}} \nc{\cls}{\operatorname{cls}}
\nc{\cof}{\operatorname{cof}}
\nc{\sSet}{\operatorname{\textbf{sSet}}}
\nc{\Map}{\operatorname{Map}}
\nc{\incl}{\operatorname{incl}}
\nc{\Hocolim}{\operatorname{Hocolim}}
\nc{\colim}{\operatorname{colim}}
\nc{\Endd}{\operatorname{End}}
\nc{\const}{\operatorname{const}}
\nc{\inn}{\operatorname{in}}
\nc{\Ev}{\operatorname{Ev}}
\nc{\rr}{\operatorname{\textbf{r}}}
\nc{\cop}{\operatorname{\textbf{l}}}
 \nc{\CCone}{\operatorname{\textbf{Cone}}}
 \nc{\dia}{\operatorname{dia}}

\numberwithin{equation}{subsection}
\newtheorem{theo}[equation]{Theorem}
\newtheorem{theor}{Theorem}
\newtheorem{lem}[equation]{Lemma}
\newtheorem{prop}[equation]{Proposition}
\newtheorem{coro}[equation]{Corollary}
\theoremstyle{definition}
\newtheorem{defi}[equation]{Definition}

\newtheorem{remk}[equation]{Remark}
\newtheorem{exmp}[equation]{Example}
\newtheorem*{defix}{Definition}
\newtheorem*{quest}{Question}
\swapnumbers
\newtheorem{numbered paragraph}[equation]{}

\newcommand{\mybox}{\ensuremath \Box}

\begin{document}

\def\S{\textbf{S}}
\def\Z{{\mathbb Z}}
\def\L{{\mathcal L}}
\def\B{{\mathcal B}}
\def\Q{{\mathcal Q}}
\def\M{{\mathcal M}}
\def\D{{\mathcal D}}
\def\R{{\mathscr R}}
\def\E{{\mathcal E}}
\def\K{{\mathcal K}}
\def\W{{\mathcal W}}
\def\N{{\mathcal N}}
\def\T{{\mathcal T}}
\def\A{{\mathcal A}}
\def\C{{\mathcal C}}
\def\P{{\mathcal P}}
\def\V{{\mathcal V}}
\def\G{{\mathcal G}}
\def\AA{{\mathscr A}}
\def\BB{{\mathscr B}}
\def\CC{{\mathscr C}}
\def\DD{{\mathscr D}}
\def\EE{{\mathscr E}}
\def\FF{{\mathscr F}}

\let\xto\xrightarrow

\title[]
{The derived category of complex periodic $K$-theory localized at an odd prime}

\author{Irakli Patchkoria}


\begin{abstract} We prove that for an odd prime $p$, the derived category $\D(KU_{(p)})$ of the $p$-local complex periodic $K$-theory spectrum $KU_{(p)}$ is triangulated equivalent to the derived category of its homotopy ring $\pi_*KU_{(p)}$. This implies that if $p$ is an odd prime, the triangulated category $\D(KU_{(p)})$ is algebraic. 
\end{abstract}

\subjclass[2010]{55P42, 18G55, 18E30}
\keywords{Derivator, model category, $K$-theory, module spectrum, stable model category, triangulated category}

\maketitle

\setcounter{section}{0}

\section{Introduction} 

\setcounter{subsection}{1}

The category of modules $\Mod KU$ over the complex periodic $K$-theory spectrum $KU$ has been studied for at least thirty years. Using the fact that the homotopy ring $\pi_*KU$ is isomorphic to the Laurent polynomial ring $\Z [u, u^{-1}]$, $|u|=2$, and the universal coefficient theorem, one can give a complete algebraic description of the homotopy category $\Ho(\Mod KU) = \D (KU)$, also known as the \emph{derived category} of $KU$. This description essentially goes back to Bousfield \cite{B85, B90}. Wolbert in \cite{W98} elaborated Bousfield's ideas in terms of modern notions of ring and module spectra. The basic idea is that the isomorphism type of any object $X \in \D(KU)$ is completely determined by the isomorphism type of the homotopy $\pi_*X$, considered as a graded $\pi_*KU$-module. Furthermore, using the universal coefficient theorem, morphisms in $\D(KU)$ admit an algebraic characterization in terms of matrices consisting of elements of $\Hom$ and $\Ext$ groups over $\pi_*KU$. Parallel to this, one can also consider the category of differential graded modules over $\pi_*KU \cong \Z [u, u^{-1}]$. The homotopy category of the latter, called the \emph{derived category} of $\pi_*KU$ and denoted by $\D(\pi_*KU)$, has very similar formal properties to those of $\D(KU)$. The isomorphism class of an arbitrary object $M \in \D(\pi_*KU)$ is determined by the isomorphism type of the homology $H_*M$, considered as a graded $\pi_*KU$-module. Moreover, the morphisms in $\D(\pi_*KU)$ can also be described in terms of matrices consisting of elements of certain $\Hom$ and $\Ext$ groups over $\pi_*KU$. 

This analogy between $\D(KU)$ and $\D(\pi_*KU)$ turns out to be a shadow of a much stronger statement. It follows from \cite[4.3.2]{G99} and \cite{W98} that there is an equivalence of categories
$$\D(KU) \sim \D(\pi_*KU).$$
In fact, the results of \cite{W98} and \cite{G99} yield many other equivalences of similar type. Let $R$ be an $A_{\infty}$-ring spectrum. Suppose that the graded global homological dimension of the homotopy ring $\pi_*R$ is at most one and assume that $\pi_*R$ is concentrated in even degrees. Then there is an equivalence of categories
$$\D(R) \sim \D(\pi_*R),$$
where $\D(R)$ is the derived category of $R$ (the homotopy category of module spectra over $R$) and $\D(\pi_*R)$ is the derived category of $\pi_*R$ (the homotopy category of differential graded modules over $\pi_*R$). Examples of such $A_{\infty}$-ring spectra include periodic complex topological $K$-theory $KU$, periodic Adams summands $E(1)$ (also known as the Johnson-Wilson spectra of height one), connective Morava $K$-theories $k(n)$ and $p$-local real periodic topological $K$-theory $KO_{(p)}$ for an odd prime $p$. We note that if the infinite loop space $\Omega^{\infty} R$ is not a product of Eilenberg-Mac Lane spaces, then the model  category $\Mod R$ is not Quillen equivalent to the model category of differential graded modules over $\pi_*R$. Hence the equivalence $\D(R) \sim \D(\pi_*R)$ does not come from a zig-zag of Quillen equivalences if $R = KU, E(1), k(n)$, or $KO_{(p)}$. 

The category $\D(R)$ has the important extra structure of a \emph{triangulated category}. The triangulated structure comes from the \emph{mapping cone sequences} (\emph{distinguished triangles})
$$\xymatrix{X \ar[r]^f & Y \ar[r] & \Cone(f) \ar[r] & \Sigma X}$$
and in particular encodes how the objects of $\D(R)$ are built out of $R$ using suspensions, desuspensions and cell attachments. Similarly, the derived category $\D(\pi_*R)$ together with the algebraic mapping cone sequences is a triangulated category (see Subsection \ref{stmodexmp} for the details). 

The question arises whether the equivalence $\D(R) \sim \D(\pi_*R)$ preserves the triangulated structures. More specifically, we are interested whether the equivalence $\D(KU) \sim \D(\pi_*KU)$ is compatible with the distinguished triangles. In order to be able to formulate these questions more precisely, we recall some terminology:

\begin{defix} Let $\T_1$ and $\T_2$ be triangulated categories. A \emph{triangulated equivalence} between $\T_1$ and $\T_2$ is an equivalence of categories
$$\xymatrix{F \colon \T_1 \ar[r]^-{\sim} & \T_2}$$
together with a natural isomorphism (called the \emph{suspension isomorphism}) 
$$\alpha \colon F \circ \Sigma \cong \Sigma \circ F$$
such that for any distinguished triangle 
$$\xymatrix{X \ar[r]^-f & Y \ar[r]^-g & Z \ar[r]^-h & \Sigma X}$$
in $\T_1$, the triangle
$$\xymatrix{F(X) \ar[r]^-{F(f)} & F(Y) \ar[r]^-{F(g)} & F(Z) \ar[rr]^-{\alpha_X \circ F(h)} & & \Sigma F(X)}$$
is distinguished in $\T_2$. Here $\alpha_X \circ F(h)$ is the composite
$$\xymatrix{F(Z) \ar[r]^{F(h)} & F(\Sigma X) \ar[r]^-{\alpha_X} & \Sigma F(X).}$$

\end{defix}

If there exists a triangulated equivalence between $\T_1$ and $\T_2$, then we say that $\T_1$ and $\T_2$ are \emph{triangulated equivalent}. 

Now we can pose the above advertised question precisely: 

\begin{quest} Is the equivalence $\D(KU) \sim \D(\pi_*KU)$ triangulated?

\end{quest}

One could also ask a weaker question: Does there exist a triangulated equivalence between $\D(KU)$ and $\D(\pi_*KU)$? The answer to any of these questions is not known. Similar questions can also be formulated when $R=E(1), k(n)$ or $KO_{(p)}$. 

The purpose of this work is to provide partial solutions to these problems. The following is Theorem \ref{mainKUtheo} in the paper and is proved at the very end:

\begin{theor} \label{theo1} Let $p$ be an odd prime. Then the derived categories $\D(KU_{(p)})$ and $\D(\pi_*KU_{(p)})$ are triangulated equivalent.
\end{theor}

Moreover, as a consequence of the general Theorem \ref{gentriangleequi} in Section 4, we also obtain:

\begin{theor} \label{theo2} Let $R$ be any of the following $A_{\infty}$-ring spectra:

\begin{itemize}

\item Periodic Adams summand $E(1)$ at an odd prime $p$;

\item The connective Morava $K$-theory spectrum $k(n)$, such that $p$ is odd or $n \neq 1$;

\item The $p$-local real periodic $K$-theory spectrum $KO_{(p)}$ for an odd prime $p$.

\end{itemize}

Then the derived categories $\D(R)$ and $\D(\pi_*R)$ are triangulated equivalent. 

\end{theor}

The equivalences constructed here are isomorphic to the ones constructed in \cite[4.3.2]{G99}. The only slight difference is that the suspension isomorphisms might deviate from each other up to a unit in $\pi_0R$.   

These theorems in particular show that the derived categories $\D(KU_{(p)})$, $\D(E(1))$, $\D(k(n))$ and $\D(KO_{(p)})$ are algebraic in the sense of \cite{Schw1}. Theorem 1 and Theorem 2 also provide new examples of \emph{exotic triangulated equivalences}, i.e., triangulated equivalences which do not come from Quillen equivalences. A well-known example of such an equivalence is the triangulated equivalence between $\D(K(n))$ and $\D(\pi_*K(n))$, where $K(n)$ is periodic Morava $K$-theory. The proof of the latter uses that $\pi_*K(n) \cong \mathbb{F}_p[v_n, v_n^{-1}]$ is a graded field and hence the graded global homological dimension of $\pi_*K(n)$ is zero. Another important example is the triangulated equivalence of stable module categories of $\Z/p^2$ and $\mathbb{F}_p[\varepsilon]/(\varepsilon^2)$. These triangulated categories are both equivalent to $\mathbb{F}_p$-vector spaces. It follows from the result of Schlichting \cite{Schl} (see also \cite{MR}) that the associated model categories are not Quillen equivalent. Further examples can be potentially constructed using the preprint \cite{F96}  of Franke (see also \cite{R08}). Franke deals with the case of the $E(n)$-local stable homotopy category for higher chromatic primes. However, one of the key steps in the proof contains a gap (partially filled in \cite{Pat12}) and this still needs to be fixed. If we take this into account, to the author's knowledge Theorem 1 and Theorem 2 provide the first examples of exotic triangulated equivalences where the underlying homological algebra is not zero-dimensional (i.e., semi-simple). 

We would also like to remind the reader that in contrast to Theorem 1 and Theorem 2, there are examples of model categories which have the rigidity property that their homotopy categories admit a unique model up to Quillen equivalence. Examples include the model categories of spectra \cite{Schw}, $K_{(2)}$-local spectra \cite{R07}, and modules over Eilenberg-Mac Lane ring spectra (see \cite{SS03}). Moreover, the paper \cite{Schw1} shows that the homotopy category of spectra, also known as the \emph{stable homotopy category}, is not algebraic. 

Finally, before describing the methods of the proofs, we briefly mention connections of Theorem 1 with $C^*$-algebras. It follows from \cite{BJS} and \cite{DEKM} that the thick subcategory $\mathbf{B}$ generated by the $C^*$-algebra $\mathbb{C}$ inside the $KK$-category (also known as the \emph{bootstrap class}) embeds into $\D(KU)$ fully-faithfully. The papers \cite{Ben} and \cite{Mah} pose the question of whether the triangulated category $\mathbf{B}$ is algebraic. Theorem 1 implies that if we $p$-localize the category $\mathbf{B}$ at an odd prime $p$, then the resulting triangulated category $\mathbf{B}_{(p)}$ is algebraic.

\subsection*{Methods of the proofs} The proof of Theorem 2 is based on the ideas of Franke \cite{F96} and the results of the paper \cite{Pat12}. We make essential use of the theory of Grothendieck derivators, which seems to be the most convenient homotopy-theoretic language for this paper. This is not surprising since none of the equivalences above come from Quillen equivalences of model categories or equivalences of $(\infty, 1)$-categories. 

The general Theorem \ref{gentriangleequi} in Section 4 which implies Theorem 2 needs the assumption that the homotopy ring $\pi_*R$ is concentrated in degrees divisible by some number $N \geq 4$. Hence Theorem \ref{gentriangleequi} does not apply to $KU$ or $KU_{(p)}$, since $\pi_mKU$ is nontrivial for every even $m$. The strategy for proving Theorem 1 is to reduce to the case of $KO_{(p)}$. More precisely, we use that
$$\iota \colon KO_{(p)} \to KU_{(p)}$$
is a $C_2$-Galois extension in the sense of \cite{Rog08}. This implies that the counit of the adjunction  
$$\xymatrix{ - \wedge_{KO_{(p)}} KU_{(p)}  \colon \D(KO_{(p)}) \ar@<0.5ex>[r] & \D(KU_{(p)}) \cocolon \iota^* \ar@<0.5ex>[l]}$$
has a natural splitting. By combining the latter with the triangulated equivalence $\D(KO_{(p)}) \sim \D(\pi_*KO_{(p)})$ for an odd prime $p$, we get Theorem 1. 

Since the homotopy groups of $KO_{(2)}$ and $KO$ are much more involved than those of $KO_{(p)}$ for an odd $p$, our methods do not apply in the $2$-local or integral case. More precisely, $KO_{(2)}$ and $KO$ have homotopy groups in odd degrees and the global homological dimension of the homotopy rings is infinite. We do not know whether the $2$-local or integral version of Theorem 1 holds. This still remains an open problem. We also do not know whether proving the $2$-local version of Theorem 1 implies the integral statement since the fracture squares are non-canonical in general triangulated categories. However, we believe that the method which will solve the problem for $KU_{(2)}$ will most likely be applicable already integrally. 

\subsection*{Acknowledgements} First of all I would like to thank Tyler Lawson and Stefan Schwede for their interest in the subject and for encouraging me to think about this problem. Special thanks go to Dustin Clausen for answering my questions in Galois theory and to Moritz Groth for tutorials on derivators. I would also like to thank Rasmus Bentmann, Jens Franke, John Greenlees, Snigdhayan Mahanta, George Nadareishvili and Constanze Roitzheim for useful conversations.   

This research was supported by the Danish National Research Foundation through the Centre for Symmetry and Deformation (DNRF92) and the Shota Rustaveli National Science Foundation grant DI/27/5-103/12.

\setcounter{subsection}{1}

\section{Preliminaries} 

\subsection{Graded homological algebra} Let $B$ be a graded ring. We will denote by $\Grmod B$ the category of graded (right) $B$-modules. This category is an abelian category and has enough projective objects. For any graded $B$-module $M$, let $\proj M$ denote the (graded) \emph{projective dimension} of $M$ in $\Grmod B$. The \emph{graded global homological dimension} of $B$, denoted by $\gl B$, is the supremum of projective dimensions of objects in $\Grmod B$.

Next, let $N \geq 2$ be an integer. We say that a graded ring $B$ is $N$-\emph{sparse} if $B$ is concentrated in degrees divisible by $N$, i.e., $B_k=0$ if $k$ is not congruent to $0$ modulo $N$. 

If $A$ is a differential graded ring, then $\Mod A$ will denote the category of differential graded modules over $A$.

\subsection{Stable model categories and triangulated categories} \label{stmodexmp}

We will freely use in this paper the language of triangulated categories (see e.g., \cite{GM96}) and model categories (see e.g., \cite{Q67, DS95, H99}). 

For a model category $\M$, we denote by $\Ho(\M)$ the \emph{homotopy category of} $\M$. Further, given objects $X$ and $Y$ of $\M$, the notation $[X,Y]$ will stand for the set of morphisms in $\Ho(\M)$ between $X$ and $Y$.  

Recall that a pointed model category $\M$ is called \emph{stable} if the \emph{suspension functor} 
$$\Sigma \colon \Ho(\M) \to \Ho(\M)$$ 
is an equivalence of categories. For any stable model category $\M$, the homotopy category $\Ho(\M)$ is triangulated with $\Sigma$ the suspension functor \cite[7.1.6]{H99}. The distinguished triangles are defined using mapping cone sequences which we recall now. For simplicity assume that $\M$ is a simplicial model category. Then the suspension can be modeled by the left derived functor of $S^1 \wedge - \colon \M \to \M$, where $S^1= I / \{0,1\}$ and $I$ denotes the standard simplicial $1$-simplex. Next, let $f \colon X \to Y$ be a map between cofibrant objects in $\M$. Consider the pushout diagram
$$\xymatrix{X \ar[r]^f  \ar[d]_{1 \wedge -} & Y \ar[d]^{j} \\ CX \ar[r] & \Cone(f), }$$
where $CX$ denotes the smash product $(I,0) \wedge X$. The natural map $(I,0) \wedge X \to S^1 \wedge X$ and the trivial map $\ast \colon Y \to S^1 \wedge X$ induce by the universal property of pushout a map
$$\pa \colon \Cone(f) \to S^1 \wedge X.$$
A triangle of the form
$$\xymatrix{X \ar[r]^f & Y \ar[r]^-{j} & \Cone(f) \ar[r]^\pa & S^1 \wedge X}$$
is called an \emph{elementary distinguished triangle}. An abstract triangle
$$\xymatrix{A \ar[r] & B \ar[r] & C \ar[r] & \Sigma A}$$
in $\Ho(\M)$ is called \emph{distinguished} (or \emph{exact}) if it isomorphic to an elementary distinguished triangle in $\Ho(\M)$. 

\begin{exmp}\label{exmpI} For any ring $k$, the category $\Ch(k)$ of unbounded chain complexes of $k$-modules with the projective model structure is a stable model category \cite[2.3.11]{H99}. The weak equivalences and fibrations in this model structure are quasi-isomorphisms (i.e., homology isomorphisms) and degreewise epimorphisms, respectively.
\end{exmp}

\begin{exmp}\label{exmpII} Let $R$ be a symmetric ring spectrum. The category $\Mod R$ of right $R$-modules admits a stable simplicial model structure \cite[Corollary 5.5.2]{HSS00}. The homotopy category Ho($\Mod R$) is called the \emph{derived category of $R$}, denoted by $\D(R)$. Note that $R$ is a compact generator of $\D(R)$ and the representable homological functor associated to $R$ is the homotopy group functor $\pi_* \colon \D(R) \longrightarrow \Grmod \pi_*R$. \end{exmp}

\begin{exmp} \label{exmpIII} Let $A$ be a differential graded ring. Then, by Example \ref{exmpI} and \cite[4.1]{SS00} (see also \cite{H97}), the category $\Mod A$ of differential graded right modules over $A$ has a projective model structure in which the weak equivalences and fibrations are as in \ref{exmpI}. This model structure is stable as well. The category Ho($\Mod A$) is called the \emph{derived category of $A$}, denoted by $\D(A)$. The differential graded module $A$ is a compact generator of $\D(A)$ and the representable homological functor associated to $A$ is the homology $H_* \colon \D(A) \longrightarrow \Grmod H_*A$.

By \cite[2.15]{S07}, there is a symmetric ring spectrum $HA$ such that $\Mod HA$ is Quillen equivalent to $\Mod A$. In particular, the derived category $\D(A)$ is triangulated equivalent to $\D(HA)$. Thus, from the point of view of model category theory Example \ref{exmpIII} is a special case of Example \ref{exmpII}. \end{exmp}

\begin{remk}\label{algcone} Note that the triangulated structure on the derived category $\D(A)$ of a differential graded ring $A$ comes from the usual shift functor and algebraic mapping cone construction. More precisely, for any $M \in \D(A)$, the suspension of $M$, denoted by $M[1]$, is given by
$$M[1]_n= M_{n-1}, \;\;\; d^{M[1]}_n= -d^{M}_{n-1}.$$
The object $M[1]$ inherits the structure of a graded right $A$-module from that of $M$ (the action of $A$ on $M[1]$ is not twisted by a sign). Furthermore, the differential $d^{M[1]}$ satisfies the Leibniz rule.

Next, let $f \colon M \longrightarrow M'$ be a morphism of differential graded modules over $A$. The \emph{(algebraic) mapping cone} $\Cone(f)$ of $f$ is a differential graded $A$-module defined by
$$(\Cone(f))_n=M_{n-1} \oplus M'_n, \;\;\; n \in \Z,$$
$$d(m,m')=(-dm, f(m)+dm').$$
The object $\Cone(f)$ comes with canonical morphisms of differential graded modules
$$\iota \colon M'\longrightarrow \Cone(f), \;\;\;\iota(m')=(0,m')$$
and
$$\partial \colon \Cone(f) \longrightarrow M[1], \;\;\; \partial(m,m')=m.$$
These morphisms together with $f \colon M \longrightarrow M'$ form a triangle
$$\xymatrix{ M \ar[r]^f & M' \ar[r]^-\iota & \Cone(f) \ar[r]^\partial & M[1],}$$
called the \emph{elementary distinguished triangle} associated to $f$.

A triangle in $\D(A)$ is distinguished if and only if it is isomorphic to an elementary distinguished triangle in $\D(A)$.

\end{remk}

\subsection{Adams spectral sequence} \label{introASS}

We start by recalling the Adams spectral sequence for general triangulated categories. It can be constructed using the methods of \cite{C98}. See also \cite[Section 2.1]{F96}. We do not provide here details as they are standard. 

\begin{prop} [The Adams spectral sequence] \label{genASS} Let $\T$ be a triangulated category and $\A$ an abelian category with a given autoequivalence $[1] \colon \A \to \A$, enough projectives and finite global homological dimension. Suppose that $F \colon \T \to \A$ is a homological functor with a natural isomorphism $F(\Sigma X) \cong F(X)[1]$. Further assume that the following two conditions hold: 

{\rm (i)} If $C$ is an object of $\T$ and $F(C)$ is projective in $\A$, then the map
$$ F \colon \Hom_{\T}(C,X) \longrightarrow \Hom_{\A}(F(C), F(X))$$
is an isomorphism for any object $X$ of $\T$.

{\rm (ii)} For any projective object $P$ in $\A$, there exists $G \in \T$ such that
$$F(G) \cong P$$
holds in $\A$.

Then there is a bounded convergent spectral sequence
$$E_2^{pq}=\Ext^p_{\A}(F(X), F(Y)[q]) \Rightarrow \Hom_{\T}(X, \Sigma^{p+q}Y).$$ 
Here $[q] \colon \A \to \A$ denotes the iteration of $[1] \colon \A \to \A$ (or of its chosen inverse).

\end{prop}

Let $\T$ be a triangulated category with infinite coproducts and $S$ a compact generator of $\T$. Consider the functor 
$$\Hom_{\T}(S,-)_* \colon \T \to \Mod \Endd_{\T}(S)_*$$ 
and the shift functor 
$$[1] \colon \Grmod \Endd_{\T}(S)_* \to \Grmod \Endd_{\T}(S)_*,$$ 
where $\Hom_{\T}(S,-)_*$ denotes the graded $\Hom$ group $\Hom_{\T}(\Sigma^*S,-)$ and $\Endd_{\T}(S)_*$ stands for the graded endomorphism ring of $S$. These satisfy the conditions of Proposition \ref{genASS} (see e.g., \cite[Lemma 5.17 and Proposition A.4]{BKS04}). Hence Proposition \ref{genASS} yields the following:

\begin{coro} \label{ASS} Let $\T$ be a triangulated category with infinite coproducts and $S$ a compact generator of $\T$, and assume that the graded global homological dimension of $\Endd_{\T}(S)_*$ is finite. Then there is a bounded convergent spectral sequence
$$E_2^{pq}=\Ext^p_{{\Endd_{\T}(S)}_*}(\Hom_{\T}(S,X)_*, \Hom_{\T}(S,Y)_*[q]) \Rightarrow \Hom_{\T}(X, \Sigma^{p+q}Y).$$

\end{coro}

In fact, this spectral sequence is a special case of a more general spectral sequence constructed in \cite{C98}.

\subsection{Franke's functor} \label{introFranke}

In this subsection we review the construction of Franke's functor. We mainly follow \cite[3.3]{Pat12} where the reader can also look up necessary details. 

Let $\M$ be a simplicial stable model category and $S$ a compact generator of $\Ho(\M)$. Having our main examples in mind, for convenience, we will denote by $\pi_*X$, the graded morphisms $[S,X]_*=[\Sigma^* S, X]$ in $\Ho(\M)$. Note that the suspension functor induces a natural isomorphism $\pi_*(\Sigma X) \cong \pi_*(X)[1]$.   

Now suppose that the graded ring $\pi_*S=[S,S]_*$ is $N$-sparse for $N \geq 2$ and the graded global homological dimension of $\pi_*S$ is less than $N$. The category $\Grmod \pi_*S$ splits as follows
$$\Grmod{\pi_*S} \sim \B \oplus \B[1] \oplus \ldots \oplus \B[N-1],$$ 
where $\B$ is the full subcategory of $\Grmod{\pi_*S}$ consisting of all those modules which are concentrated in degrees divisible by $N$. Next, let $\C_N$ stand for the poset
$$\xymatrix{\zeta_0
  \ar@{<-}[drrrr]%
    |<<<<<<<<<<<<{\text{\Large \textcolor{white}{$\blacksquare$}}}%
    |<<<<<<<<<<<<<<<<<<<<<{\text{\Large \textcolor{white}{$\blacksquare$}}}
    |>>>>>>>>>>>>>>>>>>>>>>>{\text{\Large \textcolor{white}{$\blacksquare$}}}
    |>>>>>>>>>>>>>>{\text{\Large \textcolor{white}{$\blacksquare$}}}
    |>>>>>>>>>>{\text{\Large \textcolor{white}{$\blacksquare$}}}%
& \zeta_1 & \ldots & \zeta_{N-2} & \zeta_{N-1} \\ \beta_0 \ar[u] \ar[ur] & \beta_1 \ar[u] \ar[ur] & \ldots \ar[ur] & \beta_{N-2} \ar[u] \ar[ur] & \beta_{N-1}. \ar[u] }$$
For a general diagram $X \in \M^{\C_N}$ we denote by
$$l_i \colon X_{\beta_i} \to X_{\zeta_i}, \;\;\;\;\;\;\; k_i \colon X_{\beta_{i-1}} \to X_{\zeta_i} \;\;\;\;\;\;\; i \in \Z/N\Z$$
the structure morphisms of $X$. 

Define $\L$ to be the full subcategory of $\Ho(\M^{C_N})$ consisting of those diagrams $X \in \Ho(\M^{C_N})$ which satisfy the following conditions:

\;

\;

\rm (i) The objects $X_{\beta_i}$ and $X_{\zeta_i}$ are cofibrant in $\M$ for any $i \in \Z/N\Z$;

\rm (ii) The graded $\pi_*S$-modules $\pi_*X_{\beta_i}$ and $\pi_*X_{\zeta_i}$ are objects of $\B[i]$ for any $i \in \Z/N\Z$;

\rm (iii) The map $\pi_*l_i \colon \pi_*X_{\beta_i} \to \pi_*X_{\zeta_i}$ is injective for any $i \in \Z/N\Z$.

\;

\;

In \cite[3.3]{Pat12} the category $\L$ is defined using only projectively cofibrant diagrams. The condition (i) above is weaker than being a cofibrant diagram. It allows us more flexibility with certain constructions in this paper. Since every diagram can be replaced up to weak equivalence by a projectively cofibrant diagram, this minor difference between the two definitions does not matter for the final outcome. 

Next we recall the definition of the functor $\Q \colon \L \to \Mod \pi_*S$. For any $i \in \Z/N\Z$ consider the mapping cone sequence
$$\xymatrix{X_{\beta_{i-1}} \ar[r]^{k_i} & X_{\zeta_i} \ar[r] & \Cone(k_i) \ar[r] & \Sigma X_{\beta_{i-1}}.}$$
This sequence induces a long exact sequence
$$ \ldots \to \pi_*(X_{\beta_{i-1}}) \to \pi_*(X_{\zeta_i})  \to \pi_*(\Cone(k_i)) \to \pi_*(X_{\beta_{i-1}})[1] \to \pi_*(X_{\zeta_i})[1] \to \ldots. $$
Since $X$ is an object of $\L$, it follows from the definition that $\pi_*(X_{\beta_{i-1}}) \in \B[i-1]$ and $\pi_*(X_{\zeta_i}) \in \B[i]$. Hence the maps $\pi_*(X_{\beta_{i-1}}) \to \pi_*(X_{\zeta_i})$ and $\pi_*(X_{\beta_{i-1}})[1] \to \pi_*(X_{\zeta_i})[1]$ vanish and we get a short exact sequence
$$0 \to \pi_*(X_{\zeta_{i}}) \to \pi_*(\Cone(k_i)) \to \pi_*(X_{\beta_{i-1}})[1] \to 0$$
of graded $\pi_*S$-modules for any $i \in \Z/N\Z$. After summing up we get a short exact sequence
$$0 \to \bigoplus_{i \in \Z/N\Z} \pi_*(X_{\zeta_{i}}) \to \bigoplus_{i \in \Z/N\Z} \pi_*(\Cone(k_i)) \to \bigoplus_{i \in \Z/N\Z} \pi_*(X_{\beta_{i}})[1] \to 0.$$ 
Using the map 
$$\bigoplus_{i \in \Z/N\Z} \pi_*l_i \colon \bigoplus_{i \in \Z/N\Z} \pi_*(X_{\beta_{i}}) \to \bigoplus_{i \in \Z/N\Z} \pi_*(X_{\zeta_{i}}),$$  
we can splice the latter short exact sequence with its shifted copy and get a differential 
$$d \colon \bigoplus_{i \in \Z/N\Z} \pi_*(\Cone(k_i)) \to \bigoplus_{i \in \Z/N\Z} \pi_*(\Cone(k_i))[1].$$ 
Define the functor $\Q \colon \L \to \Mod  \pi_*S$ by setting $\Q(X)= (\bigoplus_{i \in \Z/N\Z} \pi_*(\Cone(k_i)), d)$. 

Let $\K$ denote the full subcategory of $\L$ consisting of those objects $X$ which satisfy
$$\proj \pi_*X_{\beta_i} < N-1, \;\;\;\;\;\;\; \proj \pi_*X_{\zeta_i} < N-1$$
for all $i \in \Z/N\Z$. The following is proved in \cite[3.3]{Pat12} using the Adams spectral sequence (Corollary \ref{ASS}):

\begin{prop} \label{Qequi} Let $\M$ be a simplicial stable model category with a compact generator $S$ of $\Ho(\M)$. Suppose that the graded ring $\pi_*S=[S,S]_*$ is $N$-sparse for $N \geq 2$ and the graded global homological dimension of $\pi_*S$ is less than $N$. Then the restricted functor 
$$\Q\vert_{ \K} \colon \K \to \Mod  \pi_*S$$ 
is fully faithful and a differential graded module $(C,d)$ is in the essential image of $\Q\vert_{ \K}$ if and only if 
$$\proj \Ker d <N-1 \;\;\;\;\; \text{and} \;\;\;\;\; \proj \Imm d < N-1.$$ 
In particular, if $\gl \pi_*S$ is less than $N-1$, then $\Q \colon \L \to \Mod  \pi_*S$ is an equivalence of categories. 

\end{prop} 

Using this proposition we can once and for all choose an inverse $\Q^{-1} \colon \Q(\K) \to \K$. Denote by $\R'$ the composite
$$\xymatrix{ \Q(\K) \ar[rr]^-{(\Q\vert_{ \K})^{-1}} & & \K \subseteq \Ho(\M^{\C_N}) \ar[rr]^-{\Hocolim} & & \Ho(\M).}$$
By \cite[3.4]{Pat12} one has a natural isomorphism $\pi_* \circ \R' \cong H_*$. Further note that by \cite[2.2.5]{H97} every cofibrant differential graded $\pi_*S$-module is underlying projective graded module and hence $\Q(\K)$ contains all cofibrant objects. Altogether we conclude that $\R'$ gives rise to a functor
$$\R \colon \D(\pi_*S) \to \Ho(\M)$$
with a natural isomorphism $\pi_* \circ \R \cong H_*$. The main results of \cite{Pat12} show that the following hold:

\begin{itemize}

\item If $\gl \pi_*S=1$ and $N \geq 2$, then the functor $\R \colon \D(\pi_*S) \to \Ho(\M)$ is an equivalence of categories;

\item If $\gl \pi_*S=2$ and $N \geq 4$, then the functor $\R \colon \D(\pi_*S) \to \Ho(\M)$ is an equivalence of categories. 

\end{itemize}

\section{Equivalences of diagram categories and partial morphisms of derivators}

In this section we extend the results of \cite{Pat12}, recalled above, to diagram categories. We define the functor $\R$ for certain diagram categories and prove that in some cases it is an equivalence. We also show that these functors are compatible with change of diagrams. 

\subsection{The functor $\R$ for diagram categories} 

Let $\M$ be a stable model category and $S$ a compact generator of $\Ho(\M)$. Let $\P$ be a finite poset. Since the category $\P$ is both \emph{direct} and \emph{inverse} in the sense of \cite[Section 5.1]{H99}, the functor category $\M^\P$ admits the \emph{projective} and \emph{injective} model structures \cite[Theorem 5.1.3]{H99}. The weak equivalences in both model structures are the levelwise weak equivalences. The fibrations in the projective model structure are levelwise fibrations and the dual is the case for the injective model structure. For an explicit description of the projective cofibrations and injective fibrations see e.g., \cite[Section 5.1]{H99}. It follows from the definitions that the identity functor 
$$\id \colon \M^\P \to \M^\P$$
is a left Quillen functor from the projective to the injective model structure. Moreover, the identity adjunction is a Quillen equivalence. 

The functor $\pi_*=[S,-]_* \colon \M \to \Grmod \pi_*S$ induces a functor
$$\pi_* \colon \M^{\P} \to (\Grmod \pi_*S)^{\P}.$$
This functor sends weak equivalences to isomorphisms. Hence it gives rise to a homological functor on the homotopy category
$$\pi_* \colon \Ho(\M^{\P}) \to (\Grmod \pi_*S)^{\P}$$
which reflects isomorphisms. The category $(\Grmod \pi_*S)^{\P}$ is an abelian category with enough projectives (see e.g., \cite[Section 2.3]{W94}). 

The dimension $\dimm \P$ of $\P$ is defined to be the dimension of its geometric realization which coincides with the maximal length of a chain in $\P$.  It follows from \cite{Mit68} that the global homological dimension of the abelian category $(\Grmod \pi_*S)^{\P}$ is at most $\dimm \P+ \gl \pi_*S$. In particular if $\gl \pi_*S$ is finite, then the global homological dimension of $(\Grmod \pi_*S)^{\P}$ is also finite. Further we have the levelwise shift functor $[1] \colon (\Grmod \pi_*S)^{\P} \to (\Grmod \pi_*S)^{\P}$ which is an equivalence. Now using the general description of projective objects for diagram abelian categories (see e.g., \cite[Section 2.3]{W94}) and the standard idempotent splitting argument (see e.g., \cite[Proposition A.4]{BKS04}), we see that the homological functor $\pi_* \colon \Ho(\M^{\P}) \to (\Grmod \pi_*S)^{\P}$ and the latter shift satisfy the conditions of Proposition \ref{genASS}. This gives an Adams spectral sequence for computing abelian groups of morphisms in $\Ho(\M^{\P})$. 

Next, assume that $\pi_*S$ is concentrated in dimensions divisible by a natural number $N \geq 2$ and $\dimm \P + \gl \pi_*S <N$ holds. Exactly as in Subsection \ref{introFranke}, the category $(\Grmod \pi_*S)^{\P}$ splits as follows 
$$(\Grmod{\pi_*S})^{\P} \sim \B^{\P} \oplus \B^{\P}[1] \oplus \ldots \oplus \B^{\P}[N-1],$$ 
where $\B^{\P}$ is the full subcategory of $(\Grmod{\pi_*S})^{\P}$ consisting of all those $\P$-diagrams which are levelwise concentrated in degrees divisible by $N$. Note that by definition the category $\B^{\P}[i]$ is equal to the diagram category $(\B[i])^{\P}$. Now the arguments recalled in \ref{introFranke} (see  Subsections 3.3 and 3.4 of \cite{Pat12} for more details) can be applied mutatis mutandis to $\M^{\P}$ and hence we can construct the functor $\R \colon \Ho((\Mod \pi_*S)^{\P}) \to \Ho(\M^{\P})$. More precisely, let $\L^{\P}$ denote the full subcategory of $\Ho(\M^{\P \times \C_N})$ consisting of levelwise cofibrant objects which after evaluating at any element of $\P$ satisfy the conditions (ii) and (iii) from the definition of $\L$ in Subsection \ref{introFranke}. One defines a functor $\Q^{\P} \colon \L^{\P} \to (\Mod \pi_*S)^{\P}$ by levelwise repeating the construction of $\Q$. We can also consider the full subcategory $\K^{\P}$ of $\L^{\P}$ which is analogous to $\K$. However, for simplicity from now on we will assume that $\dimm \P + \gl \pi_*S < N-1$. Then $\K^{\P}=\L^{\P}$ and the same argument as in \cite[3.3]{Pat12}, using the Adams spectral sequence for $\pi_* \colon \Ho(\M^{\P}) \to (\Grmod \pi_*S)^{\P}$ (Proposition 2.3.1), shows that the functor
$$\Q^{\P} \colon \L^{\P} \to (\Mod  \pi_*S)^{\P}$$
is an equivalence of categories. We choose an inverse $(\Q^{\P})^{-1}$ once and for all so that 
$$((\Q^{\P})^{-1}, \Q^{\P})$$ 
is an adjoint equivalence of categories with $(\Q^{\P})^{-1}$ being the left adjoint. Further, the composite
$$\xymatrix{ (\Mod  \pi_*S)^{\P} \ar[rr]^-{(\Q^{\P})^{-1}} & & \L^{\P} \subseteq \Ho(\M^{\P \times \C_N}) \ar[rr]^-{\Hocolim} & & \Ho(\M^{\P})}$$  
induces the functor $\R \colon \Ho((\Mod \pi_*S)^{\P}) \to \Ho(\M^{\P})$ and there is a natural isomorphism $\pi_* \circ \R \cong H_*$. For any $\P$, we keep the same symbol $\R$ for denoting the latter functor in order to simplify notation. 

Finally, another verbatim translation of the arguments from \cite[Section 6]{Pat12} shows that the following holds:

\begin{theo} \label{diagequivcat} Suppose we are given the following data: A (simplicial) stable model category $\M$, a compact generator $S$ of $\Ho(\M)$ and a finite poset $\P$. Assume that the graded ring $\pi_*S=[S,S]_*$ is $N$-sparse for $N \geq 4$ and $\dimm \P + \gl \pi_*S \leq 2$. Then the functor
$$\R \colon \Ho((\Mod \pi_*S)^{\P}) \to \Ho(\M^{\P})$$ 
is an equivalence of categories.

\end{theo} 

\subsection{Compatibility with change of diagrams and the derivator structure} \label{paste}

In this subsection we show that the functors $\R$ are compatible with change of diagrams. Suppose we are given an order preserving map $u \colon \P_1 \to \P_2$ of posets. Then the pullback functor
$$u^* \colon M^{\P_2} \to M^{\P_1}$$
preserves weak equivalences. Further it has a left adjoint $u_{!} \colon M^{\P_1} \to M^{\P_2}$ and a right adjoint $u_* \colon M^{\P_1} \to M^{\P_2}$, called the \emph{left and right Kan extensions}. Since the functor $u^*$ preserves lewelvise notions, it follows that $(u_{!}, u^*)$ is a Quillen adjunction with respect to the projetive model structures and $(u^*, u_*)$ is a Quillen adjunction with respect to the injective model structures. These adjunctions pass to homotopy categories and we denote the derived functors with the same symbols as the point set level functors. 

Before we discuss the compatibility of $\R$ with latter change of diagram functors, we need to briefly review \emph{calculus of mates} \cite{KS74} (see also \cite[Section 1.2]{Gro13}).   
 
Consider a diagram
$$\xymatrix{
\AA    \xtwocell[1,1]{}\omit & \BB \ar[l]_{v} \\
\CC \ar[u]^{u_1}  & \DD \ar[l]^{w} \ar[u]_{u_2}
}$$
of categories. The arrow in the middle of the square indicates that we are given a natural transformation $\alpha \colon v u_2  \to u_1 w$. Suppose ${u_i}_{!}$ is left adjoint to $u_i$, $i=1,2$. Then $\alpha$ induces a natural transformation
$$\alpha_! \colon {u_1}_! v \to {u_1}_! v u_2 {u_2}_! \to {u_1}_! u_1 w {u_2}_! \to w {u_2}_!,$$
where the middle map is ${u_1}_! \alpha {u_2}_!$ and the first and last map are induced by the unit of $({u_2}_!, u_2)$ and the counit of $({u_1}_!, u_1)$, respectively. The map $\alpha_!$ is referred to as one of the possible two \emph{mates} of $\alpha$ (the other one is defined using right adjoints of $v$ and $w$ in case they exist). This construction is compatible with pasting of natural transformations. More precisely, suppose we are given a diagram of categories
$$\xymatrix{
\AA    \xtwocell[1,1]{}\omit & \BB \xtwocell[1,1]{}\omit \ar[l]_{v} &  \EE \ar[l]_{v'} \\
\CC \ar[u]^{u_1}  & \DD \ar[l]^{w} \ar[u]_{u_2} & \FF, \ar[l]^{w'} \ar[u]_{u_3}
}$$
where the arrows inside the squares are natural transformations 
$$\alpha \colon v u_2  \to u_1 w \;\;\;\; \text{and} \;\;\;\; \alpha' \colon v' u_3  \to u_2 w'.$$ 
The \emph{pasting} of these natural transformations, denoted by $\alpha \odot \alpha'$, is defined as the composite
$$\alpha \odot \alpha' \colon vv'u_3 \xto{v\alpha'} vu_2w' \xto{\alpha w'} u_1ww'.$$
Now suppose for $i=1,2, 3$, the functor $u_i$ has a left adjoint ${u_i}_{!}$. Then $(\alpha \odot \alpha')_!=\alpha'_! \odot \alpha_!$ (see e.g., \cite[Lemma 1.14]{Gro13}). Moreover, the pasting operation is associative and the vertical pasting and horizontal pasting commute with each other (see e.g., \cite{Ehr63, Ehr65, GP99}).  

In what follows we will use certain structure on the collection of diagram categories 
$$\P \mapsto \Ho(\M^{\P}).$$ 
If $\M$ is a \emph{combinatorial model category}, then the assignment 
$$\C \mapsto \Ho(\M^{\C}),$$
where $\C$ is any small category, defines a \emph{derivator} (see \cite{Cis03} and \cite[Proposition 1.30]{Gro13}). The structure of a derivator axiomatizes the derived change of diagram functors and the existence of the homotopy (derived) left and right Kan extensions. We will assume familiarity with the theory of derivators. Sources on derivators include \cite{F96, Gro13, Mal07}.

In this paper we only need the portion of the derivator structure which is indexed on finite posets. The advantage when working with posets is that one does not have to assume anything additional on the model category $\M$ to have homotopy Kan extensions 
$$u_{!} \colon \Ho(M^{\P_1}) \to \Ho(M^{\P_2}), \;\;\;\;\;\;\; u_* \colon \Ho(M^{\P_1}) \to \Ho(M^{\P_2})$$ 
for a map $u \colon \P_1 \to \P_2$ of finite posets. Since the category of finite posets satisfies the axioms of a \emph{category of diagrams} in the sense of \cite{CN08}, the theory from \cite{Gro13} applies to the derivator $\P \mapsto \Ho(\M^{\P})$ defined on finite posets (see \cite[End of Section 1.1]{Gro13}). 

Now we start establishing some compatibility results of the funtors $\R$ with change of diagrams. Let $\M$ be a (simplicial) stable model category, $S$ a compact generator of $\Ho(\M)$ and assume that the graded ring $\pi_*S=[S,S]_*$ is $N$-sparse. Further, let $\P_1$ and $\P_2$ be finite posets such that $\dimm \P_i + \gl \pi_*S < N-1$, $i=1,2$, and $u \colon \P_1 \to \P_2$ an order preserving map of posets. Then by construction we have the identity natural transformation $\id_u \colon u^*\Q^{\P_2} = \Q^{\P_1} (u \times 1)^*$. Schematically we have the diagram
$$\xymatrix{
(\Mod \pi_*S)^{\P_1}  & \L^{\P_1} \ar[l]_-{\Q^{\P_1}} \\
(\Mod \pi_*S)^{\P_2} \ar[u]^{u^*}   & \L^{\P_2}, \xtwocell[-1,-1]{}\omit \ar[l]^-{\Q^{\P_2}} \ar[u]_{(u \times 1)^* }
}$$
where the middle arrow is the identity natural transformation (i.e., the diagram commutes). Since by our choice $(\Q^{\P_i})^{-1}$ is left adjoint to $\Q^{\P_i}$, $i=1,2$, we get a mate natural transformation $(\id_u)_! \colon (\Q^{\P_1})^{-1} u^* \to (u \times 1)^* (\Q^{\P_2})^{-1}$. This natural transformation is a natural isomorphism since $((\Q^{\P_i})^{-1}, \Q^{\P_i})$ is an adjoint equivalence of categories, $i=1,2$, and $\id_u$ is an isomorphism. (Note that even if $\id_u$ is the identity, the natural transformation $(\id_u)_!$ in general is not the identity.) By pasting
$$\xymatrix{
(\Mod \pi_*S)^{\P_1} \ar[rr]^-{(\Q^{\P_1})^{-1}}   & & \L^{\P_1}  \ar[r]^-{\Hocolim}  &  \Ho(\M^{\P_1})  \\
(\Mod \pi_*S)^{\P_2} \ar[rr]_-{(\Q^{\P_2})^{-1}} \xtwocell[-1,1]{}\omit   \ar[u]^{u^*}  & & \L^{\P_2} \ar[r]^-{\Hocolim}  \xtwocell[-1,1]{}\omit \ar[u]^{(u\times 1)^*} &   \Ho(\M^{\P_2})  \ar[u]_{u^*}
}$$
and then passing to derived categories, one gets a natural transformation 
$$\gamma_u \colon \R u^* \to u^*\R.$$ 

\begin{prop} \label{change of dia} Let $\M$ be a (simplicial) stable model category, $S$ a compact generator of $\Ho(\M)$ and assume that the graded ring $\pi_*S=[S,S]_*$ is $N$-sparse. Further let $\P_1$ and $\P_2$ be finite posets such that $\dimm \P_i + \gl \pi_*S < N-1$, $i=1,2$ and $u \colon \P_1 \to \P_2$ an order preserving map of posets. Then the natural transformation $\gamma_u \colon \R u^* \to u^*\R$ is a natural isomorphism.

\end{prop}

\begin{proof} It suffices to show that the canonical natural transformation
$$\xymatrix{
\Ho(\M^{\P_1 \times \C_N}) \ar[r]^{\hspace{0.3cm} \Hocolim}   & \Ho(\M^{\P_1}) \\
\Ho(\M^{\P_2 \times \C_N}) \ar[r]^{\hspace{0.3cm} \Hocolim}  \ar[u]^{(u \times 1)^*} \xtwocell[-1,1]{}\omit & \Ho(\M^{\P_2}) \ar[u]_{u^*}
}$$
is a natural isomorphism. This follows from definitions or for example from \cite[Example 2.10]{Gro13} which shows that $\Hocolim \colon \Ho(\M^{(-) \times \C_N}) \to \Ho(\M^{(-)})$ gives a morphism of derivators. \end{proof}

\begin{prop} \label{morphism of der} Let $\M$ be a (simplicial) stable model category, $S$ a compact generator of $\Ho(\M)$ and assume that the graded ring $\pi_*S=[S,S]_*$ is $N$-sparse. 

{\rm (i)} Suppose $\P_1$, $\P_2$ and $\P_3$ are finite posets such that $\dimm \P_i + \gl \pi_*S < N-1$, $i=1,2,3$, and $u \colon \P_1 \to \P_2$ and $v \colon \P_2 \to \P_3$ are order preserving maps of posets. The the diagram
$$\xymatrix{\R (vu)^*=\R u^* v^* \ar[r]^-{\gamma_u  v^*} \ar[dr]_-{\gamma_{vu}} & u^* \R v^* \ar[d]^-{u^* \gamma_v} \\ & u^* v^* \R= (vu)^* \R }$$
commutes. 

{\rm (ii)} Suppose $\P_1$ and $\P_2$ are finite posets such that $\dimm \P_i + \gl \pi_*S < N-1$, $i=1,2$, and $u, w \colon \P_1 \to \P_2$ are order preserving maps of posets. Let $\mu \colon u \leq w$ be a natural transformation. Then the diagram 
$$\xymatrix{\R u^* \ar[r]^{\R \mu^*} \ar[d]_{\gamma_u} & \R w^* \ar[d]^{\gamma_w} \\ u^* \R \ar[r]^{\mu^* \R} & w^* \R }$$
commutes.

\end{prop}

\begin{proof} (i) Consider the diagram 
$$\xymatrix{
(\Mod \pi_*S)^{\P_1} \ar[rr]^-{(\Q^{\P_1})^{-1}} &  & \L^{\P_1}  \ar[r]^-{\Hocolim}  &  \Ho(\M^{\P_1})  \\
(\Mod \pi_*S)^{\P_2} \ar[rr]^-{(\Q^{\P_2})^{-1}} \xtwocell[-1,1]{}\omit   \ar[u]^{u^*} & & \L^{\P_2} \ar[r]^-{\Hocolim}  \xtwocell[-1,1]{}\omit \ar[u]^{(u\times 1)^*} &   \Ho(\M^{\P_2})  \ar[u]_{u^*}\\ (\Mod \pi_*S)^{\P_3} \ar[rr]^-{(\Q^{\P_3})^{-1}} \xtwocell[-1,1]{}\omit   \ar[u]^{v^*} & & \L^{\P_3} \ar[r]^-{\Hocolim}  \xtwocell[-1,1]{}\omit \ar[u]^{(v\times 1)^*} &   \Ho(\M^{\P_3}).  \ar[u]_{v^*}
}$$
Since pasting is compatible with taking mates, the pasted natural transformation  $(\id_v)_! \odot (\id_u)_!$ is equal to $(\id_{vu})_!$. Next, the right vertical pasting of the canonical natural transformations in the above diagram gives the canonical natural transformation 
$$\Hocolim ((vu \times 1)^*) \to (vu)^* \Hocolim.$$ 
This follows by a similar argument as the latter one or using the fact that 
$$\Hocolim \colon \Ho(\M^{(-) \times \C_N}) \to \Ho(\M^{(-)})$$ 
is a morphism of derivators \cite[Example 2.10]{Gro13}.

The arguments above show that by first vertically pasting and then horizontally we get the natural transformation $\gamma_{vu}$ (after passing to the derived categories). On the other hand, by definition, if we first horizontally paste and then vertically we get the natural transformation $u^* \gamma_v \circ \gamma_u  v^*$ (after passing to the derived categories). But since it does not matter in which order we paste for the final outcome, one gets the desired result.

(ii) Again by definition the diagram
$$\xymatrix{u^*\Q^{\P_2} \ar@{=}[d]_{\id_u} \ar[rrr]^{\mu^*\Q^{\P_2}} & & & w^* \Q^{\P_2} \ar@{=}[d]^{\id_w} \\ \Q^{\P_1} (u \times 1)^* \ar[rrr]^{\Q^{\P_1} (\mu \times 1)^*}  & & & \Q^{\P_1} (w \times 1)^*  }$$
commutes. This by naturality and the definition of a mate implies that the diagram 
$$\xymatrix{(\Q^{\P_1})^{-1} u^* \ar[rrr]^-{(\Q^{\P_1})^{-1} \mu^*} \ar[d]_{(\id_u)_!} & & & (\Q^{\P_1})^{-1} w^* \ar[d]^{(\id_w)_!} \\ (u \times 1)^* (\Q^{\P_2})^{-1} \ar[rrr]^-{(\mu \times 1)^* (\Q^{\P_2})^{-1}} & & & (w \times 1)^* (\Q^{\P_2})^{-1} }$$
commutes. On the other hand we also have the commutative diagram 
$$\xymatrix{\Hocolim (u \times 1)^* \ar[d] \ar[rr]^-{\Hocolim (\mu \times 1)^*} & & \Hocolim (w \times 1)^* \ar[d]\\ u^* \Hocolim \ar[rr]^-{\mu^* \Hocolim} &  & w^* \Hocolim.}$$
This can be seen by an explicit check or again using the fact that  
$$\Hocolim \colon \Ho(\M^{(-) \times \C_N}) \to \Ho(\M^{(-)})$$ 
is a morphism of derivators. Now by combining the latter two commutative diagrams and passing to the derived categories, we get the desired result. 
\end{proof}

This proposition implies that the functors 
$$\R \colon \Ho((\Mod \pi_*S)^{\P}) \to \Ho(\M^{\P})$$
almost assemble into a \emph{morphism of derivators} (see e.g., \cite[Section 2.1]{Gro13}) except that they are only defined if $\dimm \P + \gl \pi_*S < N-1$. In other words, these functors form a \emph{partial morphism of derivators}.  

\section{Triangulated structures} \label{trsec}

In this section we will show that in certain cases the functor $\R$ preserves distinguished triangles. Our strategy is to use the fact that distinguished triangles in $\Ho(\M)$ can be completely described by the derivator structure on $\Ho(\M^{(-)})$ and then apply the previous section to get the desired result. 

\subsection{Constructing distinguished triangles} \label{Constrdertriangles}

The construction of triangles in terms of derivators classically uses a poset of dimension three (see e.g., \cite{F96, Mal07} and \cite[Section 4.2]{Gro13}). Since we have restrictions on homological dimensions, it is not convenient for us to consider diagram categories indexed on posets of dimensions more than one. In this section we reformulate the construction of distinguished triangles completely in terms of one dimensional posets. Note that this reformulation is already mentioned in \cite[Remark 2 of Section 1.4]{F96}. However, no details are given there.    

First we recall the definition of the mapping cone functor. Let $\M$ be a stable model category and $[1]$ denote the poset with elements $0$ and $1$ and the relation $0 \leq 1$. Further consider the poset 
$$\xymatrix{(0,0) \ar[r] \ar[d] & (1,0) \\ (0,1)}$$
denoted by $\ulcorner$. Let $\iota \colon [1] \to \ulcorner$ be the functor (i.e., order preserving map) which sends $0$ to $(0,0)$ and $1$ to $(1,0)$. In other words, $\iota$ includes the interval $[1]$ into $\ulcorner$ as the top horizontal line. Further, let $\Box$ denote the poset
$$\xymatrix{(0,0) \ar[r] \ar[d] & (1,0) \ar[d] \\ (0,1) \ar[r] & (1,1)}$$
and $i_{\ulcorner} \colon \ulcorner \to \Box$ the obvious inclusion. Finally, we will also use the functor $j \colon [1] \to \Box$ which includes the interval $[1]$ as the right vertical arrow inside $\Box$. The \emph{mapping cone functor} $\CCone \colon \Ho(\M^{[1]}) \to \Ho(\M^{[1]})$ is then defined as the composite (see e.g., \cite[Definition 3.18]{Gro13}):
$$\xymatrix{\Ho(\M^{[1]}) \ar[r]^{\iota_*} & \Ho(\M^{\ulcorner}) \ar[r]^{(i_{\ulcorner})_{!}} & \Ho(\M^\Box) \ar[r]^{j^*} &  \Ho(\M^{[1]}).}$$
This definition of the mapping cone functor uses the poset $\Box$ which has dimension two. We would like to reformulate the latter definition in terms of one dimensional diagrams. This is possible since the functor $j \colon [1] \to \Box$ has a left adjoint $p \colon \Box \to [1]$ which sends the top vertices of $\Box$ to $0$ and the lower vertices of $\Box$ to $1$. Indeed, the functor $j^* \colon \Ho(\M^{\Box}) \to \Ho(\M^{[1]})$ is left adjoint to $p^* \colon \Ho(M^{[1]}) \to \Ho(\M^{\Box})$. Let $\omega \colon \ulcorner \to [1]$ denote the composite $p \circ i_{\ulcorner}$. Then the homotopy left Kan extension $\omega_{!} \colon \Ho(\M^{\ulcorner}) \to \Ho(\M^{[1]})$ is naturally isomorphic to the composite of the left adjoints of $(i_{\ulcorner})^*$ and $p^*$. But the latter is isomorphic to the composite $j^* (i_{\ulcorner})_{!}$ and hence we can redefine the cone functor $\CCone \colon \Ho(\M^{[1]}) \to \Ho(\M^{[1]})$ to be the composite:
$$\xymatrix{\Ho(\M^{[1]}) \ar[r]^{\iota_{*}} & \Ho(\M^{\ulcorner}) \ar[r]^{\omega_{!}} & \Ho(\M^{[1]}).}$$
From this point on by the cone functor we will always mean the latter composite. 

Next, we recall the definition of the \emph{suspension functor} $\Sigma \colon \Ho(\M) \to \Ho(\M)$. Consider the functors $\inn_0, \inn_1 \colon \ast \to [1]$ which include the point as $0$ and $1$ respectively into $[1]$. Then the suspension functor is defined as the composite:
$$\xymatrix{\Ho(\M) \ar[r]^{(\inn_0)_*} & \Ho(\M^{[1]}) \ar[r]^{\CCone} & \Ho(\M^{[1]}) \ar[r]^{\inn_1^*} & \Ho(\M).}$$ 
It follows immediately from the definitions that this functor is naturally isomorphic to the suspension functor defined in \cite[7.1.6]{H99}. 

Next, we recall the definition of the \emph{underlying arrow functor} (see e.g., \cite{Gro13}): The unique natural transformation $\inn_0 \to \inn_1$ induces a natural transformation $\inn_0^* \to \inn_1^*$, where $\inn_0^*$ and $\inn_1^*$ are functors from $\Ho(\M^{[1]})$ to $\Ho(\M)$. These functors and the natural transformation assemble into a single functor $\dia_{[1]} \colon \Ho(\M^{[1]}) \to \Ho(\M)^{[1]}$. It follows for example from \cite[Lemma 1.9]{Gro13} that $\dia_{[1]}$ is full and essentially surjective. 
In particular, given an object $f \in \Ho(\M^{[1]})$, one obtains a morphism 
$$\xymatrix{\inn_0^*(f) \ar[r]^{\dia_{[1]}(f)} & \inn_1^*(f) }$$
and every other morphism in $\Ho(\M)$ is isomorphic to such a morphism. Hence we can define distinguished triangles in $\Ho(\M)$ by only using the morphisms of type $\dia_{[1]}(f)$ for some object $f$ in $\Ho(\M^{[1]})$. 

We are now almost ready to recall the definition of distinguished triangles in terms of the derivator structure of $\Ho(\M^{(-)})$. For this we need to observe that the target of $\dia_{[1]}(f)$ is naturally isomorphic to the source of $\CCone(f)$. This follows by definitions. However, since this isomorphism is built in in the definition of the distinguished triangles, we would like to make it more explicit. The isomorphism is given in terms of restriction and homotopy Kan extension functors. As advertised, the diagrams involved in the construction will be only one dimensional. 

Let $\inn_{(1,0)} \colon \ast \to \ulcorner$ denote the inclusion of the vertex $(1,0)$ into $\ulcorner$. The diagram
$$\xymatrix{
\ast \ar@{=}[d] \ar[r]^{\inn_{(1,0)}}  \xtwocell[1,1]{}\omit & \ulcorner \ar[d]^{\omega} \\
\ast \ar[r]_{\inn_0}  & [1] }$$
commutes and the arrow in the middle just indicates the identity natural transformation. This diagram induces a commutative diagram of homotopy categories
$$\xymatrix{
\Ho(\M^{[1]})  \ar[d]_{\omega^*} \ar[r]^{\inn_{0}^*}   & \Ho(\M) \ar@{=}[d] \\
\Ho(\M^{\ulcorner})  \ar[r]_{\inn_{(1,0)}^*}  & \Ho(\M),  \xtwocell[-1,-1]{}\omit
}$$
where the middle arrow again stands for the identity transformation. If we form the mate $(\id)_{!}$ of this identity natural transformation, then we get by definition the natural transformation
$$\xymatrix{\inn_{(1,0)}^* \ar[rr]^-{\inn_{(1,0)}^* \eta_{\omega}} & & \inn_{(1,0)}^*{\omega}^*{\omega}_{!} = \inn_{0}^*{\omega}_{!},}$$
where $\eta_{\omega}$ is the unit of the adjunction $({\omega}_{!}, {\omega}^*)$. 

Next, consider the adjunction
$$\xymatrix{ \iota^* \colon \Ho(\M^{\ulcorner}) \ar@<0.5ex>[r] & \Ho(\M^{[1]}) \cocolon \iota_* \ar@<0.5ex>[l]},$$
induced by the functor $\iota \colon [1] \to \ulcorner$. The functor $\iota$ is fully faithful. Hence, by \cite[Proposition 1.20]{Gro13} the counit $\varepsilon_{\iota} \colon \iota^* \iota_* \to \id$ is an isomorphism. 

Finally, we are ready to define a natural transformation $\zeta \colon \inn_1^* \to \inn_0^* \CCone$ as the composite
$$\xymatrix{\inn_1^* \ar[rr]^-{\inn_1^* (\varepsilon_{\iota})^{-1}} & & \inn_1^* \iota^* \iota_*=\inn_{(1,0)}^* \iota_* \ar[rr]^-{\inn_{(1,0)}^* \eta_{\omega} \iota_*} & & \inn_0^* \omega_{!} \iota_*}.$$ 

\begin{lem} \label{Conesource} The natural transformation $\zeta \colon \inn_1^* \to \inn_0^* \CCone$ is an isomorphism.

\end{lem}

\begin{proof} The commutative diagram
$$\xymatrix{
\ast \ar@{=}[d] \ar[r]^{\inn_{(1,0)}}  \xtwocell[1,1]{}\omit & \ulcorner \ar[d]^{\omega} \\
\ast \ar[r]_{\inn_0}  & [1] }$$
is obtained by the pasting 
$$\xymatrix{
\ast    \xtwocell[1,1]{}\omit \ar[r]^{\inn_{1}} \ar@{=}[d] & [1] \ar[r]^{\iota} \ar[d] \xtwocell[1,1]{}\omit &  \ulcorner \ar[d]^{\omega} \\
\ast \ar@{=}[r]  & \ast \ar[r]_{\inn_0} & [1],
}$$
where the natural transformations are the identities. Hence the induced diagram on homotopy categories
$$\xymatrix{
\Ho(\M^{[1]})  \ar[d]_{\omega^*} \ar[r]^{\inn_{0}^*}   & \Ho(\M) \ar@{=}[d] \\
\Ho(\M^{\ulcorner})  \ar[r]_{\inn_{(1,0)}^*}  & \Ho(\M)  \xtwocell[-1,-1]{}\omit
}$$
is also equal to the pasting
$$\xymatrix{
\Ho(\M^{[1]})  \ar[d]_{\omega^*} \ar[r]^{\inn_{0}^*}   & \Ho(\M) \ar[d]^{\const} \ar@{=}[r] & \Ho(\M) \ar@{=}[d] \\
\Ho(\M^{\ulcorner})  \ar[r]_{\iota^*}  & \Ho(\M^{[1]}) \ar[r]_{\inn_1^*}  \xtwocell[-1,-1]{}\omit & \Ho(\M). \xtwocell[-1,-1]{}\omit
}$$
It follows from \cite[Proposition 1.18]{Gro13} that the mate $(\id)_{!}$ associated to the second square is an isomorphism. Indeed, the functor $\inn_1 \colon \ast \to [1]$ is right adjoint and hence the square involving $\inn_1$ is \emph{homotopy exact} in the sense of \cite[Definition 1.15]{Gro13}. The mate $(\id)_{!}$  coming from the first square in the latter diagram is also an isomorphism because the square 
$$\xymatrix{
[1] \ar[r]^{\iota} \ar[d] \xtwocell[1,1]{}\omit &  \ulcorner \ar[d]^{\omega} \\
\ast \ar[r]_{\inn_0} & [1],
}$$
is a $2$-pullback. Indeed, this square is associated to the \emph{comma category} $(\omega/\inn_0)$ and hence by \cite[Proposition 1.26]{Gro13} homotopy exact, which by definition means that the corresponding mate $(\id)_{!}$ is an isomorphism. Since forming mates and pasting commute with each other, we conclude that the mate associated to the diagram 
$$\xymatrix{
\Ho(\M^{[1]})  \ar[d]_{\omega^*} \ar[r]^{\inn_{0}^*}   & \Ho(\M) \ar@{=}[d] \\
\Ho(\M^{\ulcorner})  \ar[r]_{\inn_{(1,0)}^*}  & \Ho(\M)  \xtwocell[-1,-1]{}\omit
}$$
is an isomorphism. As we have already observed, the latter is exactly the natural transformation $\inn_{(1,0)}^* \eta_{\omega}$. This implies the desired result since the counit $\varepsilon_\iota \colon \iota^* \iota_* \to \id$ is a natural isomorphism.  \end{proof}

Finally, we are ready to redefine the distinguished triangles. In order to be consistent we have to once and for all fix models for the homotopy left and right Kan extensions for the derivator $\Ho(\M^{(-)})$. The homotopy Kan extensions are only uniquely defined up to canonical isomorphisms, since adjoints for a given functor are. Any such choice of Kan extensions leads to a triangulated structure on $\Ho(\M)$. The triangulated category structures on $\Ho(\M)$ associated to different choices of adjoints are equivalent as triangulated categories via the identity functor and the canonical isomorphism between different suspension functors. 

\begin{defi} \label{triangl} Let $f$ be an object of $\Ho(\M^{[1]})$. An \emph{elementary homotopy triangle} associated to $f$ is defined to be the sequence in $\Ho(\M)$
$$\xymatrix{\inn_0^*(f) \ar[rr]^-{\dia_{[1]}(f)} & & \inn_1^*(f) \ar[rrr]^-{\dia_{[1]} (\CCone(f)) \circ \zeta(f)} & & & \inn_1^*(\CCone(f)) \ar[rrr]^-{\inn_1^*(\CCone(\eta_{\inn_0}(f)))}   & & & \Sigma (\inn_0^*(f)), }$$
where we use the isomorphism $\zeta(f) \colon \inn_1^*(f) \cong \inn_0^*(\CCone(f))$, the definition of the suspension functor $\Sigma= \inn_1^* \CCone (\inn_0)_*$ and the unit $\eta_{\inn_0}$ of the adjunction $(\inn_0^*, (\inn_0)_*)$.

\end{defi}

If $\M$ is a simplicial stable model category or the model category of differential graded modules over a differential graded ring (see Subsection \ref{stmodexmp}), then the choices of the left and right homotopy Kan extensions can be done in such a way that a triangle in $\Ho(\M)$ is distinguished in the sense of Subsection \ref{stmodexmp} if and only if it is isomorphic in $\Ho(\M)$ to an elementary homotopy triangle. The following is a consequence of \cite[Theorem 4.16]{Gro13} and \cite[Section 7.1]{H99} and summarizes the above discussion:

\begin{prop} \label{Dertriang} Suppose $\M$ is any of the stable model categories from Subsection \ref{stmodexmp}. 

{\rm (i)} For any fixed choice of homotopy Kan extensions (for relevant one dimensional diagrams), the suspension functor and the elementary homotopy triangles define a triangulated category structure on $\Ho(\M)$. A triangle in $\Ho(\M)$ is distinguished with respect to this structure if and only if it is isomorphic in $\Ho(\M)$ to an elementary homotopy triangle. Up to triangulated equivalence this triangulated structure does not depend on the choice of homotopy Kan extensions.

{\rm (ii)} One can choose necessary homotopy Kan extensions in such a way that the associated triangulated structure on $\Ho(\M)$ coming from the elementary homotopy triangles coincides with the one defined in \cite[Section 7.1]{H99}. 

\end{prop}

From now on we will once and for all fix the choices and the triangulated structure from Proposition \ref{Dertriang} (ii) on $\Ho(\M)$. In particular, when checking that a functor $F \colon \Ho(\M) \to \Ho(\N)$ is triangulated, it suffices to show that $F$ preserves up to isomorphism of triangles the elementary homotopy triangles. This observation will be used in the next subsection.   

\subsection{Preserving distinguished triangles} \label{preservetr} In this subsection we assume that $\M$ is a simplicial stable model category with a compact generator $S$ of $\Ho(\M)$ such that $\pi_*S=[S,S]_*$ is $N$-sparse. We will also assume that $N \geq 4$ and $\gl \pi_*S=1$. Our main example which satisfies these conditions is the model category $\Mod KO_{(p)}$ of modules over periodic real $K$-theory localized at an odd prime $p$.

By Theorem \ref{diagequivcat}, for any poset $\P$ whose dimension is at most one, the functor
$$\R \colon \Ho((\Mod \pi_*S)^{\P}) \to \Ho(\M^{\P})$$ 
is an equivalence of categories. Further, Proposition \ref{change of dia} and Proposition \ref{morphism of der} tell us that these equivalences are compatible with functors between posets of dimension at most one and their natural transformations. This together with the description of distinguished triangles in the previous subsection will imply that the functor
$$\R \colon \D(\pi_*S)=\Ho(\Mod \pi_*S) \to \Ho(\M)$$ 
is a triangulated equivalence of triangulated categories. 

The first thing to show is that the $\R$ functors are compatible with the underlying arrow functors $\dia_{[1]}$. Let $\alpha \colon \inn_0 \to \inn_1$ denote the unique natural transformation of the functors $\inn_0, \inn_1 \colon \ast \to [1]$. It follows from Proposition \ref{morphism of der} (ii) that the diagram
$$\xymatrix{\R\inn_0^* \ar[r]^{\R \alpha^*} \ar[d]_{\gamma_{\inn_0}} & \R \inn_1^* \ar[d]^{\gamma_{\inn_1}} \\ \inn_0^* \R \ar[r]^{\alpha^*\R} & \inn_1^* \R }$$
commutes. Here in the upper arrow the symbol $\R$ stands for the functor $\R \colon \Ho(\Mod \pi_*S) \to \Ho(\M)$ and in the lower arrow for the functor $\R \colon \Ho((\Mod \pi_*S)^{[1]}) \to \Ho(\M^{[1]})$. By definition of the functor $\dia_{[1]}$, the latter commutative diagram implies that for any $f \in \Ho((\Mod \pi_*S)^{[1]})$, the diagram 
$$\xymatrix{\R(\inn_0^*(f)) \ar[rr]^{\R(\dia_{[1]}(f))} \ar[d]_{\gamma_{\inn_0}(f)} & & \R (\inn_1^*(f)) \ar[d]^{\gamma_{\inn_1}(f)} \\ \inn_0^* (\R(f)) \ar[rr]^{\dia_{[1]}(\R(f))} & & \inn_1^* (\R(f)) }$$
commutes. Recall also that by Proposition \ref{change of dia} the vertical arrows are isomorphisms in this diagram. We will prove bellow that in fact the elementary homotopy triangle associated to an object $f$ of $\Ho((\Mod \pi_*S)^{[1]}) $ is mapped under $\R$ to a triangle isomorphic to the elementary homotopy triangle associated to $\R(f) \in \Ho(\M^{[1]}) $. Since elementary homotopy triangles were defined using homotopy Kan extensions, we have to check that the $\R$ functors commute with relevant homotopy Kan extensions. 

The following proposition follows for example from \cite[Proposition 2.9]{Gro13}. We still provide details here because compared to the standard derivator literature, we have restrictions on dimensions of diagrams.    

\begin{prop} \label{commutKan} Let $\M$ be a simplicial stable model category and $S$ a compact generator of $\Ho(\M)$. Suppose that $\pi_*S=[S,S]_*$ is $N$-sparse for some $N \geq 4$ and $\gl \pi_*S=1$. Further, let $u \colon \P_1 \to \P_2$ be an order preserving map of at most one dimensional posets. Then the following hold:

{\rm (i)} The mate $\beta_u=(\gamma_u)_{!} \colon u_{!} \R \to \R u_{!}$ is a natural isomorphism. (The same applies to $\gamma_u^{-1}$ and the right homotopy Kan extensions and hence we also have a natural isomorphism $\delta_{u} \colon \R u_* \to u_*\R$.)

{\rm (ii)} The functors $\R$ are compatible with units and counits of homotopy Kan extensions. E.g., the diagram
$$\xymatrix{u_{!} \R u^* \ar[d]_{u_{!} \gamma_u} \ar[rr]^{\beta_u u^*} & & \R u_{!} u^* \ar[d]^{\R \varepsilon_{u}} \\ u_{!}u^* \R \ar[rr]^{\varepsilon_{u}\R} & & \R  }$$ 
commutes and similar diagrams for other counit and units also commute. \end{prop}

\begin{proof} (i) Recall that $\beta_u=(\gamma_u)_{!}$ is defined as the composite
$$\xymatrix{u_{!} \R \ar[rr]^{u_{!} \R \eta_{u}} & & u_{!} \R u^* u_{!} \ar[rr]^{u_{!}\gamma_u u_{!}} & & u_{!}u^* \R u_{!} \ar[rr]^{\varepsilon_{u} \R u_{!} } & & \R u_{!}, }$$
where $\eta_u$ is the unit. The functors $\R$ are equivalences of categories by Theorem \ref{diagequivcat}. One can choose the inverse $\R^{-1}$ and consider $\R$ as a left adjoint functor to $\R^{-1}$. Then forming the second mate associated to the natural isomorphism $\gamma_u \colon \R u^* \to u^* \R$ gives a natural isomorphism $u^* \R^{-1} \to \R^{-1} u^*$. This implies the desired result since $\beta_u$ is conjugate to the latter natural transformation (see e.g., \cite[Lemma 1.14]{Gro13}).

(ii) This part follows from the commutative diagram:
$$\xymatrix{& u_{!} \R u^* u_{!} u^* \ar[d]^{u_{!} \R u^* \varepsilon_u} \ar[rr]^{u_{!} \gamma_u u_{!} u^*} & & u_{!} u^* \R u_{!} u^* \ar[rr]^{\varepsilon_{u} \R u_{!} u^*} & & \R u_{!} u^* \ar[d]^{\R \varepsilon_u} \\ u_{!} \R u^* \ar[ur]^{u_{!}\R \eta_{u} u^*} \ar@{=}[r] & u_{!} \R u^* \ar[rr]^{u_{!} \gamma_u} & & u_{!} u^* \R \ar[rr]^{\varepsilon_u \R}  & & \R. }$$
The right hand square commutes since the lower horizontal composition is a natural transformation. The left hand triangle commutes because of one of the triangular identities of the adjunction $(u_{!}, u^*)$. \end{proof}

Now recall that in the previous subsection the mapping cone functor was defined as the composite $\omega_{!} \iota_*$. Since $[1]$ and $\ulcorner$ are one dimensional posets, Proposition \ref{commutKan} (i) applies to $\iota$ and $\omega$. Hence we conclude that there is a canonical isomorphism 
$$\xymatrix{c \colon \R \CCone \ar[r]^-{\cong} & \CCone \R.}$$
This isomorphism is obtained by combining the canonical isomorphisms $\delta_{\iota} \colon \R\iota_* \cong \iota_*\R$ and $\beta_{\omega} \colon \omega_{!} \R \cong \R \omega_{!} $ coming from Proposition \ref{commutKan} (i). The following is a corollary of Proposition \ref{morphism of der} (i) and Proposition \ref{commutKan} (ii):

\begin{coro} \label{sourcecone} The diagram (consisting of natural isomorphisms)
$$\xymatrix{\R \inn_1^* \ar[d]_{\gamma_{\inn_1}} \ar[r]^-{\R \zeta} & \R \inn_0^* \CCone \ar[d]^{(\inn_0^* c) \circ (\gamma_{\inn_0} \CCone) } \\ \inn_1^* \R \ar[r]^-{\zeta \R} & \inn_0^* \CCone \R}$$
commutes. 

\end{coro}

\begin{proof} The natural isomorphism $\zeta$ is defined in terms of the homotopy Kan extensions. This implies the desired result. \end{proof}

Another corollary of Proposition \ref{commutKan} is that the functor 
$$\R \colon \D(\pi_*S)=\Ho(\Mod \pi_*S) \to \Ho(\M)$$ 
commutes with suspensions. Indeed, define $a \colon \R \Sigma \to \Sigma \R$ to be the following composite:
$$\xymatrix{\hspace{-1.0cm} \R \inn_1^* \CCone (\inn_0)_* \ar[rr]^-{\gamma_{\inn_1} (\CCone (\inn_0)_*)} & & \inn_1^* 
\R \CCone (\inn_0)_* \ar[rr]^-{\inn_1^* (c (\inn_0)_*) } & & \inn_1^* 
\CCone \R (\inn_0)_* \ar[rr]^-{\inn_1^* 
\CCone (\delta_{\inn_0})} & &   \inn_1^* \CCone (\inn_0)_* \R. }$$
All the arrows in this composite are isomorphisms and hence so is $a \colon \R \Sigma \to \Sigma \R$. 

\;

\;

\;

Finally, we are ready to show that the functor $\R \colon \D(\pi_*S)=\Ho(\Mod \pi_*S) \to \Ho(\M)$ is triangulated. This is the main result of the section:

\begin{theo} \label{gentriangleequi} Let $\M$ be a simplicial stable model category and $S$ a compact generator of $\Ho(\M)$. Suppose that $\pi_*S=[S,S]_*$ is $N$-sparse for some $N \geq 4$ and $\gl \pi_*S=1$. Then the functor
$$\R \colon \D(\pi_*S)=\Ho(\Mod \pi_*S) \to \Ho(\M)$$
together with the natural isomorphism $a \colon \R \Sigma \cong \Sigma \R$ is a triangulated equivalence.

\end{theo}

\begin{proof} That the functor $\R$ is an equivalence of categories is for example proved in \cite[Theorem 3.1.4]{Pat12} (see Subsection \ref{introFranke}). We will now show that it is compatible with distinguished triangles. By Proposition \ref{Dertriang} it suffices to check that $\R$ preserves elementary homotopy triangles. More precisely, the desired result will follow if we prove that for any $f \in \Ho((\Mod \pi_*S)^{[1]})$, the following diagram commutes:
{\footnotesize
$$\xymatrix{\R(\inn_0^*(f)) \ar[rr]^-{\R(\dia_{[1]}(f))} \ar[d]_{\gamma_{\inn_0}(f)} & & \R(\inn_1^*(f)) \ar[d]^{\gamma_{\inn_1}(f)} \ar[rrr]^-{\R(\dia_{[1]} (\CCone(f)) \circ \zeta(f))} & & & \R(\inn_1^*(\CCone(f))) \ar[rrr]^-{a \circ \R(\inn_1^*(\CCone(\eta_{\inn_0}(f))))} \ar[d]^{(\inn_1^* c(f)) \circ (\gamma_{\inn_1} (\CCone(f))) } & & & \Sigma(\R(\inn_0^*(f))) \ar[d]^{\Sigma (\gamma_{\inn_0}(f)) } \\ \inn_0^*(\R(f)) \ar[rr]^-{\dia_{[1]}(\R(f))} & & \inn_1^*(\R(f))  \ar[rrr]^-{\dia_{[1]}( \CCone(\R(f))) \circ \zeta(\R(f))} & & & \inn_1^*(\CCone(\R(f))) \ar[rrr]^-{\inn_1^*(\CCone(\eta_{\inn_0}(\R(f))))} & & & \Sigma (\inn_0^*(\R(f))). }$$}
Recall that the vertical morphisms are isomorphisms and the lower sequence is the elementary homotopy triangle associated to $\R(f) \in \Ho(\M^{[1]})$. 

We start by reminding the reader that we have already checked that the left square commutes. Indeed, this is just the compatibility of the $\R$ functors with $\dia_{[1]}$. Next, we observe that the diagram
$$\xymatrix{\R(\inn_0^*(\CCone(f))) \ar[d]_{\gamma_{\inn_0}(\CCone(f))} \ar[rrr]^{\R(\dia_{[1]}(\CCone(f)))} & & & \R(\inn_1^*(\CCone(f))) \ar[d]^{\gamma_{\inn_1}(\CCone(f))} \\ \inn_0^*(\R(\CCone(f))) \ar[d]_{\inn_0^*(c(f))} \ar[rrr]^{\dia_{[1]}(\R(\CCone(f)))} & & & \inn_1^*(\R (\CCone(f))) \ar[d]^{\inn_1^*(c(f))} \\ \inn_0^*(\CCone(\R(f))) \ar[rrr]^{\dia_{[1]}(\CCone(\R(f)))} & & & \inn_1^*(\CCone(\R(f))) }$$
commutes. The upper square commutes because of the compatibility between $\R$ and $\dia_{[1]}$. The lower square commutes since the morphism $\inn_0^* \to \inn_1^*$ is natural. The latter commutative diagram together with Corollary \ref{sourcecone} implies that the middle square in the main comparison diagram is commutative.

Finally, we have to argue why the right hand square commutes. This follows from the commutative diagram
$$\xymatrix{\R(\inn_1^*(\CCone(f))) \ar[rrr]^-{\R(\inn_1^*(\CCone(\eta_{\inn_0}(f))))} \ar[d]_-{\inn_1^*(c(f)) \circ \gamma_{\inn_1} } & & & \R(\inn_1^* (\CCone((\inn_0)_*(\inn_0^*(f))))) \ar[d]^-{\inn_1^*(c((\inn_0)_*(\inn_0^*(f)))) \circ \gamma_{\inn_1}} \\ \inn_1^*(\CCone(\R(f))) \ar[d]_-{\inn_1^*(\CCone(\eta_{\inn_0}(\R(f))))} \ar[rrr]^-{\inn_1^*(\CCone(\R(\eta_{\inn_0}(f))))} & & & \inn_1^* (\CCone(\R((\inn_0)_*(\inn_0^*(f))))) \ar[d]^-{\inn_1^*(\CCone(\delta_{\inn_0}(\inn_0^*(f))))} \\ \inn_1^*(\CCone((\inn_0)_*(\inn_0^*(\R(f))))) & & & \inn_1^*(\CCone((\inn_0)_*(\R(\inn_0^*(f))))), \ar[lll]^-{\Sigma(\gamma_{\inn_0}(f))} }$$  
where for short $\gamma_{\inn_1}$ on the left stands for $\gamma_{\inn_1}(\CCone(f))$ and on the right for $\gamma_{\inn_1}(\CCone((\inn_0)_*(\inn_0^*(f))))$. The upper square commutes since $\inn_1^*(c) \circ \gamma_{\inn_1} (\CCone)$ is a natural transformation and the lower one because of Proposition \ref{commutKan} (ii). \end{proof}

\subsection{Examples}

In this subsection we list some examples to which Theorem \ref{gentriangleequi} applies.

\begin{exmp} \label{oddko}  Let $p$ be an odd prime. Consider the $p$-local real periodic $K$-theory spectrum $KO_{(p)}$. The spectrum $KO_{(p)}$ admits an $A_{\infty}$-structure (even an $E_{\infty}$-structure). The homotopy ring $\pi_*KO_{(p)}$ is isomorphic to the graded Laurent polynomial ring $\Z_{(p)}[v, v^{-1}]$, where the element $v$ has degree $4$. Hence the graded global homological dimension of $\pi_*KO_{(p)}$ is equal to 1 and it is concentrated in degrees divisible by $4$. This implies that the model category $\Mod KO_{(p)}$ together with the compact generator $KO_{(p)} \in \D(KO_{(p)})$ satisfies the conditions of Theorem \ref{gentriangleequi} (see Example \ref{exmpII}). We conclude that the functor
$$\R \colon \D(\Z_{(p)}[v, v^{-1}]) \to \D(KO_{(p)})$$ 
realizes a triangulated equivalence. Finally, we note that the model categories $\Mod KO_{(p)}$ and $\Mod \Z_{(p)}[v, v^{-1}]$ are not Quillen equivalent, since the infinite loop space $BO_{(p)}$ of $KO_{(p)}$ is not a product of Eilenberg-MacLance spaces (see e.g., \cite[Corollary A.1.3]{Pat12}). \end{exmp}

\begin{exmp} \label{morava} Let $p$ be a prime and $n$ a positive integer. Assume that either $p$ is odd or $n \neq 1$. Consider the connective Morava $K$-theory spectrum $k(n)$ at the prime $p$. According to \cite{L03}, the spectrum $k(n)$ admits an $A_{\infty}$-structure. The homotopy ring $\pi_*k(n)$ is isomorphic to the graded polynomial algebra $\mathbb{F}_{p}[v_n]$, where the degree of $v_n$ is equal to $2(p^n-1)$. This implies that $\Mod k(n)$ and the compact generator $k(n) \in \D(k(n))$ satisfy the conditions of Theorem \ref{gentriangleequi}. Hence we conclude that if either $p$ is odd or $n \neq 1$, then the functor 
$$\R \colon \D(\mathbb{F}_{p}[v_n]) \to \D(k(n))$$
is a triangulated equivalence. If $p=2$ and $n=1$, then the functor $\R$ is an equivalence of categories (\cite[Example 5.3.2]{Pat12}), but it still remains open whether $\R$ is triangulated.  

We again note that this triangulated equivalence does not come from a Quillen equivalence (see e.g., Example 5.3.2 of \cite{Pat12}). 

\end{exmp} 

\begin{exmp} \label{JohnsonWilson} Let $p$ be an odd prime again. Then the Johnson-Wilson spectrum $E(1)$ at the prime $p$ (also known as the periodic Adams summand) gives another example to which Theorem \ref{gentriangleequi} applies. Indeed, we know that $E(1)$ admits an $A_{\infty}$-structure (see e.g., \cite{L03}) and we have an isomorphism of graded rings
$$\pi_*E(1) \cong \Z_{(p)}[v_1, v_1^{-1}], \;\;\;\;\;\;$$
where $v_1$ has degree $2(p-1)$. Hence, by Theorem \ref{gentriangleequi}, the functor 
$$\R \colon \D(\Z_{(p)}[v_1, v_1^{-1}]) \to \D(E(1))$$
is a triangulated equivalence. Once again we observe that the corresponding model categories $\Mod \Z_{(p)}[v_1, v_1^{-1}]$ and $\Mod E(1)$ are not Quillen equivalent (see \cite[Example 4.3.2]{Pat12}). 

\end{exmp}

\section{The derived category of $KU_{(p)}$ for an odd prime $p$}

In this section we prove the main result. We show that for any odd prime $p$, the equivalence 
$$\R \colon \D(\pi_*KU_{(p)}) \to \D(KU_{(p)})$$ 
is triangulated. Note that Theorem \ref{gentriangleequi} cannot be applied directly to $\Mod KU_{(p)}$, since $\pi_*KU_{(p)} \cong \Z_{(p)}[u, u^{-1}]$, $|u|=2$, is only $2$-sparse. The idea is to descend to $KO_{(p)}$ whose homotopy groups are $4$-sparse, apply Theorem \ref{gentriangleequi} to $\Mod KO_{(p)}$ and then ascend back to $KU_{(p)}$. The fact that $p$ is odd is essentially used in our argument here, since $\pi_*KO_{(2)}$ has infinite homological dimension and Theorem \ref{gentriangleequi} does not apply to $\Mod KO_{(2)}$. Whether the equivalence $\R \colon \D(\pi_*KU_{(2)}) \to \D(KU_{(2)})$ is triangulated remains still open.  

\subsection{The Galois extension $KO_{(p)} \to KU_{(p)}$}

We will work in this section in the convenient model category of commutative algebra spectra in the sense of \cite{S04}. The main point of this model structure is that a cofibration of commutative $R$-algebras over a commutative symmetric ring spectrum $R$, is also a cofibration of underlying $R$-modules. This model structure is also known with different names like the $R$-model structure or the positive flat model structure. In what follows the words cofibrant and cofibration will refer to the notions in this model structure.  

Fix a cofibrant commutative symmetric ring spectrum model for $KO_{(p)}$. The complexification map gives a map of commutative symmetric ring spectra $\iota \colon KO_{(p)} \to KU_{(p)}$. Without loss of generality we may assume that this map is a cofibration of commutative $KO_{(p)}$-algebras, or equivalently $KU_{(p)}$ is a cofibrant commutative $KO_{(p)}$-algebra. 

The commutative ring map $\iota \colon KO_{(p)} \to KU_{(p)}$ is in fact a $C_2$-\emph{Galois extension} in the sense of \cite{Rog08} (see \cite[Proposition 5.3.1]{Rog08}). We do not really fully need this fact but rather just a consequence of it. By Lemma 9.1.2 of \cite{Rog08}, every Galois extension of commutative ring spectra with a discrete Galois group is a \emph{separable extension} and thus $\iota \colon KO_{(p)} \to KU_{(p)}$ is a separable extension. The latter means that the multiplication map
$$\mu \colon KU_{(p)} \wedge_{KO_{(p)}} KU_{(p)} \to KU_{(p)}$$
which is a map of $KU_{(p)} \wedge_{KO_{(p)}} KU_{(p)}$-modules has a section $\sigma \colon KU_{(p)} \to KU_{(p)} \wedge_{KO_{(p)}} KU_{(p)} $ in the derived category $\D(KU_{(p)} \wedge_{KO_{(p)}} KU_{(p)} )$. This implies the following:

\begin{prop} \label{splitcounit} The counit of the adjunction
$$\xymatrix{ - \wedge_{KO_{(p)}} KU_{(p)}  \colon \D(KO_{(p)}) \ar@<0.5ex>[r] & \D(KU_{(p)}) \cocolon \iota^* \ar@<0.5ex>[l]}$$
has a natural splitting in $\D(KU_{(p)})$. In other words, for any cofibrant $KU_{(p)}$-module $X$, the counit
$$ \varepsilon_X \colon X \wedge_{KO_{(p)}} KU_{(p)} \to X $$
has a section $s_X \colon X \to X \wedge_{KO_{(p)}} KU_{(p)}$ in $\D(KU_{(p)})$ (i.e., $\varepsilon_X s_X=1$ in $\D(KU_{(p)})$) and $s_X$ is natural with respect to morphisms in $\D(KU_{(p)})$.
\end{prop}

\begin{proof} As already indicated above, this is a formal consequence of the fact that the extension $\iota \colon KO_{(p)} \to KU_{(p)}$ is separable. Indeed, for any cofibrant $KU_{(p)}$-module $X$, we have a functor
$$X \wedge_{KU_{(p)}} - \colon  \D(KU_{(p)} \wedge_{KO_{(p)}} KU_{(p)} ) \to \D(KU_{(p)}).$$ 
Observe that there are canonical isomorphisms 
$$X \wedge_{KU_{(p)}} (KU_{(p)} \wedge_{KO_{(p)}} KU_{(p)}) \cong X \wedge_{KO_{(p)}} KU_{(p)}$$ 
and 
$$X \wedge_{KU_{(p)}} KU_{(p)} \cong X$$ 
in $\D(KU_{(p)})$. From here one can see that the functor $X \wedge_{KU_{(p)}} -$ sends the multiplication map $\mu \colon KU_{(p)} \wedge_{KO_{(p)}} KU_{(p)} \to KU_{(p)}$ up to a natural isomorphism to the counit 
$$\varepsilon_X \colon X \wedge_{KO_{(p)}} KU_{(p)} \to X.$$ 
Define $s_X$ as the composite
$$\xymatrix{X \cong X \wedge_{KU_{(p)}} KU_{(p)} \ar[rr]^-{X \wedge_{KU_{(p)}}\sigma} & & X \wedge_{KU_{(p)}} (KU_{(p)} \wedge_{KO_{(p)}} KU_{(p)}) \cong X \wedge_{KO_{(p)}} KU_{(p)}}$$
Since $\sigma$ splits $\mu$, we get the identity $\varepsilon_X s_X=1$ in $\D(KU_{(p)})$.  \end{proof}

\begin{coro} \label{retracttr} A triangle
$$\xymatrix{X \ar[r]^-f & Y \ar[r]^-g & Z \ar[r]^-h & \Sigma X}$$
in $\D(KU_{(p)})$ is distinguished if and only if 
the triangle 
$$\xymatrix{\iota^*(X) \ar[r]^-{\iota^*(f)} & \iota^*(Y) \ar[r]^-{\iota^*(g)} & \iota^*(Z) \ar[r]^-{\iota^*(h)} & \iota^*(\Sigma X) = \Sigma \iota^*(X)}$$
in $\D(KO_{(p)})$ is distinguished. 

\end{coro}

\begin{proof} Since the functor $\iota^*$ is exact, one part of the claim is straightforward. The other part follows from Proposition \ref{splitcounit} together with the facts that $- \wedge_{KO_{(p)}} KU_{(p)}$ is exact and a retract of a distinguished triangle is distinguished. 

\end{proof}

\begin{remk} We note that the proof of the above proposition can be applied to any separable extension. This implies that for any separable extension, the derived counit of the change of scalars adjunction has a natural splitting. 

\end{remk}

\subsection{Proof of the main result} \label{mainresultsubsec}

Let $p$ be an odd prime throughout. In this subsection we show that the diagram
$$\xymatrix{\D(\pi_*KU_{(p)}) \ar[d]_-{\iota^*} \ar[r]^-{\R} & \D(KU_{(p)}) \ar[d]^-{\iota^*} \\ \D(\pi_*KO_{(p)}) \ar[r]^-{\R} & \D(KO_{(p)})}$$
commutes up to a natural isomorphism which is compatible with suspensions. This will imply the desired result by Corollary \ref{retracttr} since the vertical functors are well-known to be triangulated and the lower horizontal functor is triangulated by Theorem \ref{gentriangleequi} (see Example \ref{oddko}). 

Before showing that the diagram commutes, we will make the suspension isomorphisms more explicit. The natural isomorphism
$$a \colon \R \Sigma \cong \Sigma \R$$
for $\R \colon \D(\pi_*KO_{(p)}) \to \D(KO_{(p)})$ was constructed using the derivator structure. It was essential that we could extend $\R$ to a partial morphism of derivators defined on one dimensional diagrams. However, one cannot use the same strategy to construct a suspension isomorphism for the functor
$\R \colon \D(\pi_*KU_{(p)}) \to \D(KU_{(p)})$ 
since $\pi_*KU_{(p)}$ is only $2$-sparse and the results of Section \ref{trsec} do not apply to this case. Instead we have to employ the suspension isomorphism described in \cite[Proposition 3.5.1]{Pat12} which works in more general settings. This isomorphism is constructed as follows: Let $\M$ be a simplicial stable model category with a compact generator $S \in \Ho(\M)$. Suppose that the graded ring $\pi_*S=[S,S]_*$ is $N$-sparse for $N \geq 2$ and the graded global homological dimension of $\pi_*S$ is less than $N$. Under these conditions, we have Franke's functor
$$\R \colon \D(\pi_*S) \to \Ho(\M)$$
and the natural isomorphism $\pi_* \circ \R \cong H_*$ (see Subsection \ref{introFranke}). The full subcategory $\K$ of $\Ho(\M^{C_N})$ has an autoequivalence $\# \colon \K \to \K$. The functor $\#$ suspends and rotates once an object of $\K$. More precisely, 
$$ X^\#_{ \beta_i} = S^1 \wedge X_{\beta_{i-1}}, \;\;\;\;\;\; X^\#_{ \zeta_i} = S^1 \wedge X_{\zeta_{i-1}},$$
$$ k^\#_i= S^1 \wedge k_{i-1}, \;\;\; l^\#_i= S^1 \wedge l_{i-1}, \;\;\; i \in \Z/N\Z.$$
By the construction of the functor $\Q \colon \K \to \Mod \pi_*S$, the suspension isomorphism for $\pi_*(-)$ gives a natural isomorphism
$$\xymatrix{\xi' \colon [1] \circ \Q \ar[r]^-{\cong} & \Q \circ \#.}$$
On the other hand, we have the canonical natural isomorphism
$$\xymatrix{\xi'' \colon \Hocolim \circ \# \ar[r]^-{\cong} & \Sigma \circ \Hocolim}$$
coming from the definition of the homotopy colimit. Now recall that we made a choice of the inverse $\Q^{-1} \colon \Q(\K) \to \K$. This choice is done in such a way that $\Q^{-1}$ is left adjoint to $\Q$ and the unit  
$$\xymatrix{\eta \colon \id \ar[r]^-{\cong} & \Q \Q^{-1}}$$
and the counit
$$\xymatrix{\varepsilon \colon \Q^{-1} \Q \ar[r]^-{\cong} & \id}$$
are fixed once and for all. Fixing these choices allows us to form the mate
$$(\xi')_{!} \colon \Q^{-1} \circ [1] \rightarrow \# \circ \Q^{-1}$$
which is a natural isomorphism. Pasting $(\xi')_{!}$ and $\xi''$
$$\xymatrix{
\Q(\K) \ar[r]^-{\Q^{-1}} \ar[d]_{[1]}   & \K  \ar[r]^-{\tiny{\Hocolim}} \ar[d]^{\#}  &  \Ho(\M) \ar[d]^{\Sigma}  \\
\Q(\K) \ar[r]^-{\Q^{-1}}      & \K  \ar[r]^-{\tiny{\Hocolim}} \xtwocell[-1,-1]{}\omit     &  \Ho(\M)  \xtwocell[-1,-1]{}\omit  
}$$
we get a natural transformation $\xi'' \odot (\xi')_{!}$ which induces a natural isomorphism
$$\xymatrix{\xi \colon \R \circ [1]  \ar[r]^-{\cong} &  \Sigma \circ \R}$$
after passing to the derived categories. 

If we apply this construction to $\Mod KU_{(p)}$ and $\Mod KO_{(p)}$, we get natural isomorphisms $\xi_{KU} \colon \R \circ  [1] \cong \Sigma \circ \R$ and $\xi_{KO} \colon \R \circ [1] \cong \Sigma \circ \R$. On the other hand, we also have the natural isomorphism
$$a \colon \R \Sigma \cong \Sigma \R$$
for $KO_{(p)}$, constructed in Subsection \ref{preservetr}. By Theorem \ref{gentriangleequi}, the functor 
$$\R \colon \D(\pi_*KO_{(p)}) \to \D(KO_{(p)})$$ 
together with the natural isomorphism $a \colon \R \Sigma \cong \Sigma \R$ is a triangulated equivalence. In order to be able to show that the equivalence
$$\R \colon \D(\pi_*KU_{(p)}) \to \D(KU_{(p)})$$ 
is triangulated, we have to relate the natural transformation $\xi_{KU}$ with the natural transformation $a$. We will see below how $\xi_{KU}$ corresponds to $\xi_{KO}$ under the restriction functor
$$\iota^* \colon \D(KU_{(p)}) \rightarrow  \D(KO_{(p)}).$$
This leaves us with the task to see a connection between $\xi_{KO}$ and $a$. By Proposition \ref{Dertriang} (ii) the suspensions defined in Subsection \ref{stmodexmp} and the derivator-theoretic suspensions from Subsection \ref{Constrdertriangles} coincide. In other words, we have 
$$\Sigma=[1] \colon \D(\pi_*KO_{(p)}) \rightarrow \D(\pi_*KO_{(p)})$$  
and 
$$\Sigma = S^1 \wedge -   \colon \D(KO_{(p)}) \rightarrow \D(KO_{(p)}).$$
We remind the reader that this can be achieved by appropriately choosing the left and right homotopy Kan extensions in the derivator structure of given model categories. Now once again one can choose appropriately $\Q^{-1}$ in the construction of $\R$ and show explicitly that $\xi_{KO}$ and $a$ coincide. However, this is another lengthy calculation which will overload the paper with unnecessary technicalities. Instead we show using much more conceptual methods that the natural transformations $\xi_{KO}$ and $a$ only differ up to a fixed unit in $\mathbb{Z}_{(p)}$. This is a weaker statement but it is enough for our purposes. The desired result about the connection between  $\xi_{KO}$ and $a$ now immediately follows from the proposition below which we believe is of independent interest.
\begin{theo} \label{selfequiKO} Let $p$ be an odd prime. The automorphism group (of not necessarily triangulated natural automorphisms) of the identity functor $1 \colon \D(KO_{(p)}) \rightarrow \D(KO_{(p)})$ is isomorphic to the group of units $\mathbb{Z}_{(p)}^{\times}$.  
\end{theo} 

\begin{proof} The proof is based on the fact that the graded global homological dimension of $\pi_* KO_{(p)}$ is equal to $1$. By Corollary \ref{ASS} we have the universal coefficient exact sequence for any $KO_{(p)}$-modules $X$ and $Y$:
$$ \xymatrix{ 0 \ar[r] & \Ext^{1}_{\pi_*KO_{(p)}}( \pi_*X[1], \pi_*Y) \ar[r] & [X,Y] \ar[r]^-{\pi_*} & \Hom_{\pi_*KO_{(p)}} (\pi_*X, \pi_*Y) \ar[r] & 0.}$$
The universal coefficient exact sequence implies that for any object $X \in \D(KO_{(p)})$, there is a splitting in $\D(KO_{(p)})$
$$X \cong X_0 \vee X_1 \vee X_2 \vee X_3,$$ 
where for any $i$, the graded module $\pi_*X_i$ is concentrated in degrees congruent to $i$ modulo $4$. Given now a natural automorphism $\lambda \colon 1 \to 1$ of the identity functor, the strategy is to show that it is constant for objects whose homotopy groups are concentrated in degrees congruent to a fixed $i$ modulo $4$. Then we have to show that these constants coincide for different congruence classes modulo $4$. 

By adjunction the endomorphism ring $[KO_{(p)}, KO_{(p)}]_0$ is isomorphic to the ring $\pi_0 KO_{(p)}$. The latter is isomorphic to $\mathbb{Z}_{(p)}$. Hence any automorphism of $KO_{(p)}$ in $\D(KO_{(p)})$ is given by a unit in $\mathbb{Z}_{(p)}$. In particular, the automorphism $\lambda_{KO_{(p)}} \colon KO_{(p)} \to KO_{(p)}$ is equal to $c_0= c_0 \cdot 1_{KO_{(p)}}$ for some $c_0 \in \mathbb{Z}_{(p)}^{\times}$. We will now show that for any $KO_{(p)}$-module $X$, the morphism $\lambda_X$ is equal to $c_0= c_0 \cdot 1_{X}$. 

First suppose that $\pi_*X$ is concentrated in degrees divisible by $4$. The universal coefficient formula implies that
$$\pi_* \colon [X,X] \rightarrow \Hom_{\pi_*KO_{(p)}}(\pi_*X, \pi_*X)$$
is an isomorphism. Since $\pi_*KO_{(p)} \cong \mathbb{Z}_{(p)}[v, v^{-1}]$, where $v$ has degree $4$, it follows that the group $\Hom_{\pi_*KO_{(p)}}(\pi_*X, \pi_*X)$ is isomorphic to $\Hom_{\mathbb{Z}}(\pi_0X, \pi_0X)$. Hence, one concludes that
$$\pi_0 \colon [X,X] \rightarrow \Hom_{\mathbb{Z}}(\pi_0X, \pi_0X)$$ 
is an isomorphism and in order to show that $\lambda_X=c_0$, it suffices to prove that $\pi_0(\lambda_X)=c_0$. This follows from the naturality of $\lambda$. Indeed, take any $x \in \pi_0X$. We can represent $x$ by a morphism $x \colon KO_{(p)} \rightarrow X$ (denoted by the same symbol $x$) in $\D(KO_{(p)})$. Since $\lambda$ is a natural transformation, the diagram
$$\xymatrix{\pi_0 KO_{(p)} \ar[r]^-{x_*} \ar[d]_{(\lambda_{KO_{(p)}})_*} & \pi_0X \ar[d]^{(\lambda_X)_*} \\  \pi_0 KO_{(p)} \ar[r]^-{x_*} & \pi_0X}$$
commutes. We compute
$$(\lambda_X)_*(x)= (\lambda_X)_*(x_*(1))= x_*( (\lambda_{KO_{(p)}})_*(1))= x_*(c_0)= c_0 \cdot x.$$
This shows that $\lambda_X=c_0$ if $X$ is such that $\pi_*X$ is concentrated in degrees divisible by $4$. Similarly, using the fact that for any $i \in \{1,2,3\}$ the endomorphism ring of $\Sigma^i KO_{(p)}$ is isomorphic to $\pi_0 KO_{(p)}=\mathbb{Z}_{(p)}$ and the fact that $\pi_i$ is represented by $\Sigma^i KO_{(p)}$, we can conclude that $\lambda_X=c_i$, for a constant $c_i \in  \mathbb{Z}_{(p)}^{\times}$, if $\pi_*X$ is concentrated in degrees congruent to $i$ modulo $4$. 

Next, one needs to show that all the constants are the same, i.e., 
$$c_0=c_1=c_2=c_3.$$
We will only show that $c_0=c_1$ since the other identities are proved similarly. Suppose we are given $X_0$ and $X_1$, such that $\pi_*X_i$ is concentrated in degrees congruent to $i$ modulo $4$ and $i=0,1$. Then by naturality, for any morphism $f \colon X_0 \to X_1$ in $\D(KO_{(p)})$, the diagram
$$\xymatrix{X_0 \ar[r]^{f} \ar[d]_{c_0=\lambda_{X_0}} & X_1 \ar[d]^{c_1=\lambda_{X_1}} \\ X_0 \ar[r]^{f} & X_1}$$
commutes. Hence $(c_0 - c_1)f=0$ and thus we conclude that the number $c_0 - c_1$ annihilates $[X_0, X_1]$. The universal coefficient exact sequence and periodicity tells us that $[X_0, X_1]$ is isomorphic to $\Ext^{1}_{\Z}(\pi_0X_0, \pi_1X_1)$, implying that $c_0 - c_1$ annihilates the $\Ext$ group $\Ext^{1}_{\Z}(\pi_0X_0, \pi_1X_1)$. Let $\frac{m}{n}=c_0-c_1$, for $m, n \in \Z$. The numerator $m$ also annihilates $\Ext^{1}_{\Z}(\pi_0X_0, \pi_1X_1)$. We will now show that $m$ is infinitely divisible by $p$ which will imply that $m=0$ and hence $c_0 - c_1=0$.

Consider the $KO_{(p)}$-module $M(p^l) \wedge KO_{(p)}$, for any $l \geq 1$, where $M(p^l)$ is the mod $p^l$ Moore spectrum. Then 
the homotopy groups of $M(p^l) \wedge KO_{(p)}$ are concentrated in degrees divisible by $4$. One has,
$$\Ext^{1}_{\Z}(\pi_0(M(p^l) \wedge KO_{(p)}), \pi_1(\Sigma  KO_{(p)})) \cong \Ext^{1}_{\Z}(\Z/p^l\Z, \Z_{(p)}) \cong \Z/p^l\Z.$$
Thus the numerator $m$ of $c_0-c_1$ annihilates $\Z/p^l\Z$ for any $l$. Hence $m$ is infinitely divisible by $p$, implying that it is zero.

Let $c$ denote the common value $c_0=c_1=c_2=c_3$. Take any $KO_{(p)}$-module $X$ and consider the splitting
$$X \cong X_0 \vee X_1 \vee X_2 \vee X_3,$$
described at the beginning of the proof. By naturality, the diagram
$$\xymatrix{X_0 \vee X_1 \vee X_2 \vee X_3 \ar[r]^-{\cong} \ar[d]_{\lambda_{X_0} \vee \lambda_{X_1} \vee \lambda_{X_2} \vee \lambda_{X_3}} &  X \ar[d]^{\lambda_X} \\  X_0 \vee X_1 \vee X_2 \vee X_3 \ar[r]^-{\cong} & X }$$
commutes. Here the upper and lower isomorphisms are the same. The previous paragraph tells us that $\lambda_{X_i}=c$ for $i=0,1,2,3$. Hence $\lambda_X=c$. 
 
\end{proof}

\begin{remk} The reader will notice from the proof that the latter theorem is much more general than it is presented here. First of all, there is no need to restrict to automorphisms. Any endomorphism of the identity functor is also given by multiplication on a (not necessarily invertible) number. Further, a similar theorem holds for more general ring spectra than $KO_{(p)}$. For example, it also holds for $KU_{(p)}$ or $KU$. We will not use any of these observations in this paper and therefore do not go into details. 

\end{remk}

\begin{prop} \label{KUKOcom} Let $p$ be an odd prime. The diagram of categories
$$\xymatrix{\D(\pi_*KU_{(p)}) \ar[d]_-{\iota^*} \ar[r]^-{\R} & \D(KU_{(p)}) \ar[d]^-{\iota^*} \\ \D(\pi_*KO_{(p)}) \ar[r]^-{\R} & \D(KO_{(p)})}$$
commutes up to a natural isomorphism.
\end{prop} 

\begin{proof} We remind the reader that we have chosen the map $\iota \colon KO_{(p)} \to KU_{(p)}$ so that it is a cofibration of commutative $KO_{(p)}$-algebras in the sense of \cite{S04}. In particular, $KU_{(p)}$ is cofibrant when considered as a $KO_{(p)}$-module via the map $\iota$. For convenience, we denote from now on the functor 
$$\R \colon \D(\pi_*KU_{(p)}) \to \D(KU_{(p)})$$
by $\R_{KU}$ and the functor
$$\R \colon \D(\pi_*KO_{(p)}) \to \D(KO_{(p)})$$
by $\R_{KO}$. Let us recall the constructions of $\R_{KU}$ and $\R_{KO}$: We have two compositions 
$$\xymatrix{ \Q(\K_{KU}) \ar[rr]^-{\Q^{-1}} & & \K_{KU} \subseteq \Ho((\Mod KU_{(p)})^{\C_2}) \ar[rr]^-{\Hocolim} & & \D(KU_{(p)})}$$
and
$$\xymatrix{ \Q(\K_{KO}) \ar[rr]^-{\Q^{-1}} & & \K_{KO} \subseteq \Ho((\Mod KO_{(p)})^{\C_4}) \ar[rr]^-{\Hocolim} & & \D(KO_{(p)}),}$$
denoted by $\R'_{KU}$ and $\R'_{KO}$, respectively. Here $\K_{KU}$ is the full subcategory of $\Ho((\Mod KU_{(p)})^{\C_2})$ consisting of diagrams $X \in \Ho((\Mod KU_{(p)})^{\C_2})$ satisfying the conditions (i), (ii) and (iii) from Subsection \ref{introFranke} and for which additionally $\pi_*KU_{(p)}$-modules $\pi_*(X_{\zeta_0})$, $\pi_*(X_{\zeta_1})$, $\pi_*(X_{\beta_0})$ and $\pi_*(X_{\beta_1})$ are projective. Similarly the full subcategory $\K_{KO}$ of $\Ho((\Mod KO_{(p)})^{\C_4})$ consists of those diagrams $Y$ which satisfy the conditions (i), (ii) and (iii) from Subsection \ref{introFranke} and for which $\pi_*(Y_{\zeta_i})$ and $\pi_*(Y_{\beta_i})$, $i=0,1,2,3,$ are projective $\pi_*KO_{(p)}$-modules. 

The functors $\R'_{KU}$ and $\R'_{KO}$ induce $\R_{KU}$ and $\R_{KO}$ after passing to the derived categories. (Strictly speaking this construction of $\R_{KO}$ formally differs from the one we were using before, since in the original construction we do not restrict to those diagrams $Y$ for which $\pi_*(Y_{\zeta_i})$ and $\pi_*(Y_{\beta_i})$, $i=0,1,2,3,$ are projective $\pi_*KO_{(p)}$-modules. However, this does not make any difference and the resulting functors between derived categories are the same. The reason is that the essential image of $\Q$ in both cases contains all cofibrant differential graded $\pi_*KO_{(p)}$-modules. Hence the localization of the essential image of $\Q$ gives the full derived category $\D(\pi_*KO_{(p)})$ in both cases.) 

In order to be able to compare $\R_{KU}$ and $\R_{KO}$, we have to slightly modify the construction of $\R_{KO}$ and replace it by an isomorphic functor whose construction will use the diagram $\C_2$ rather than $\C_4$. 

The homotopy ring $\pi_*KO_{(p)}$ is concentrated in degrees divisible by $4$. In particular, it is concentrated in degrees divisible by $2$. We can apply the construction of Subsection \ref{introFranke} to $KO_{(p)}$ by taking only into account that the ring $\pi_*KO_{(p)}$ is $2$-sparse. This gives us a composite
$$\xymatrix{\R''_{KO} \colon \Q(\K_{KO}^{(2)}) \ar[rr]^-{\Q^{-1}} & & \K_{KO}^{(2)} \subseteq \Ho((\Mod KO_{(p)})^{\C_2}) \ar[rr]^-{\Hocolim} & & \D(KO_{(p)}),}$$
where $\K_{KO}^{(2)}$ is the full subcategory consisting of diagrams $Z \in \Ho((\Mod KO_{(p)})^{\C_2})$ satisfying (i), (ii) and (iii) from Subsection \ref{introFranke} and for which additionally $\pi_*KO_{(p)}$-modules $\pi_*(Z_{\zeta_0})$, $\pi_*(Z_{\zeta_1})$, $\pi_*(Z_{\beta_0})$ and $\pi_*(Z_{\beta_1})$ are projective. After passing to the derived category the functor $\R''_{KO}$ induces a functor
$$\R_{KO}^{(2)} \colon \D(\pi_*KO_{(p)}) \to \D(KO_{(p)}).$$
The functor $\R_{KO}^{(2)}$ is naturally isomorphic to $\R_{KO}$. To see this we construct a functor
$$F \colon \K_{KO} \to \K_{KO}^{(2)}.$$
Given $X \in \K_{KO}$ with structure maps $l_i \colon X_{\beta_i} \to X_{\zeta_i}$ and $k_i \colon X_{\beta_{i-1}} \to X_{\zeta_i}$, $i \in \mathbb{Z}/4\mathbb{Z}$, define 
$$F(X)_{\zeta_0} = X_{\zeta_0} \vee X_{\zeta_2}, \;\;\;\; F(X)_{\zeta_1} = X_{\zeta_1} \vee X_{\zeta_3}, $$ 
$$F(X)_{\beta_0} = X_{\beta_0} \vee X_{\beta_2}, \;\;\;\; F(X)_{\beta_1} = X_{\beta_1} \vee X_{\beta_3}, $$
with structure maps 
$$l_0^{F(X)} = l_0 \vee l_2 \colon F(X)_{\beta_0} \to  F(X)_{\zeta_0},$$
$$l_1^{F(X)} = l_1 \vee l_3 \colon F(X)_{\beta_1} \to  F(X)_{\zeta_1},$$
$$k_1^{F(X)} = k_1 \vee k_3 \colon F(X)_{\beta_0} \to  F(X)_{\zeta_1},$$
and
$$\xymatrix{k_0^{F(X)} \colon F(X)_{\beta_1}= X_{\beta_1} \vee X_{\beta_3} \ar[r]^-{k_2 \vee k_0} & X_{\zeta_2} \vee X_{\zeta_0} \cong  X_{\zeta_0} \vee X_{\zeta_2} = F(X)_{\zeta_0}.}$$
By definition, we have $\Q \cong \Q \circ F$ and $\Hocolim \circ F \cong \Hocolim$. Now by first forming a mate of the isomorphism $\Q \cong \Q \circ F$ and then pasting one gets
$$\xymatrix{
\Q(\K_{KO}^{(2)}) \ar[r]^-{\Q^{-1}}   & \K_{KO}^{(2)}  \ar[r]^-{\Hocolim}  &  \D(KO_{(p)})  \\
\Q(\K_{KO}) \ar[r]^-{\Q^{-1}} \xtwocell[-1,1]{}\omit   \ar@{=}[u]  & \K_{KO} \ar[r]^-{\Hocolim}  \xtwocell[-1,1]{}\omit \ar[u]^{F} &     \D(KO_{(p)}). \ar@{=}[u]
}$$
We remind the reader that the inverses $\Q^{-1}$ are chosen so that $(\Q^{-1}, \Q)$ are adjoint equivalences with $\Q^{-1}$ being the left adjoint. Since both natural transformations involved in the latter pasting are natural isomorphisms, one obtains the desired natural isomorphism $\R_{KO}^{(2)} \cong \R_{KO}$. Hence in order to proof the proposition it suffices to show that the diagram
$$\xymatrix{\D(\pi_*KU_{(p)}) \ar[d]_-{\iota^*} \ar[r]^-{\R_{KU}} & \D(KU_{(p)}) \ar[d]^-{\iota^*} \\ \D(\pi_*KO_{(p)}) \ar[r]^-{\R_{KO}^{(2)}} & \D(KO_{(p)})}$$
commutes up to a natural isomorphism. 

The ring homomorphism $\iota \colon \pi_*KO_{(p)} \to \pi_*KU_{(p)}$ makes $\pi_*KU_{(p)}$ into a free graded $\pi_*KO_{(p)}$-module. This implies that the functor $\iota^* \colon \Mod \pi_*KU_{(p)} \to \Mod \pi_*KO_{(p)}$ sends the underlying projective differential graded modules to the underlying projective ones. Since the essential images $\Q(\K_{KO}^{(2)}) \subset \Mod \pi_*KO_{(p)}$ and $\Q(\K_{KU}) \subset \Mod \pi_*KU_{(p)}$ consist exactly from these kind of differential graded modules (see \cite[3.3]{Pat12}), we see that the functor $\iota^*$ restricts to a functor
$$\iota^* \colon \Q(\K_{KU}) \to \Q(\K_{KO}^{(2)}).$$
Now let $\theta \colon \iota^* \Q \cong \Q \iota^*$ denote the obvious natural isomorphism which can be constructed directly using the definitions. Further, $\kappa \colon \Hocolim \iota^* \cong \iota^* \Hocolim$ stands for the canonical natural isomorphism coming from the fact that the functor $\iota^* \colon \D(KU_{(p)}) \to \D(KO_{(p)})$ preserves homotopy colimits. The pasting
$$\xymatrix{
\Q(\K_{KU}) \ar[d]_{\iota^*} \ar[r]^-{\Q^{-1}}   & \K_{KU} \ar[d]^{\iota^*} \ar[r]^-{\Hocolim} & \D(KU_{(p)}) \ar[d]^{\iota^*} \\
 \Q(\K_{KO}^{(2)}) \ar[r]_-{\Q^{-1}}  & \K_{KO}^{(2)} \ar[r]_-{\Hocolim}  \xtwocell[-1,-1]{}\omit & \D(KO_{(p)}). \xtwocell[-1,-1]{}\omit
}$$
gives a natural isomorphism $\kappa \odot \theta_{!} \colon \R''_{KO} \iota^* \cong \iota^* \R'_{KU}$. After passing to the derived categories we get the desired natural isomorphism $\R_{KO}^{(2)} \iota^* \cong \iota^* \R_{KU}$. \end{proof}
 
Next, we would like to see that the natural isomorphism described in Proposition \ref{KUKOcom} is compatible in an appropriate sense with the suspension isomorphisms. Recall again that by Theorem \ref{gentriangleequi}, the equivalence 
$$\R_{KO} \colon \D(\pi_*KO_{(p)}) \to \D(KO_{(p)})$$ 
together with the natural isomorphism $a \colon \R_{KO} \circ [1] \cong \Sigma \circ \R_{KO}$ is a triangulated equivalence. By conjugating with the natural isomorphism $\R_{KO}^{(2)} \cong \R_{KO}$, one gets a natural isomorphism 
$$a^{(2)} \colon \R_{KO}^{(2)} \circ [1] \cong \Sigma \circ \R_{KO}^{(2)}.$$
This natural isomorphism makes the functor $\R_{KO}^{(2)} \colon \D(\pi_*KO_{(p)}) \to \D(KO_{(p)})$ into a triangulated equivalence. On the other hand, we also have the natural isomorphism 
$$\xymatrix{\xi_{KO} \colon \R_{KO}^{(2)} \circ [1]  \ar[r]^-{\cong} &  \Sigma \circ \R_{KO}^{(2)}}$$
described at the beginning of Subsection \ref{mainresultsubsec}. Theorem \ref{selfequiKO} implies that there exists an invertible number $u \in \mathbb{Z}_{(p)}^{\times}$ such that $u \xi_{KO} = a^{(2)}$. Define $s \colon \R_{KU} \circ [1] \cong \Sigma \circ \R_{KU}$ to be the natural isomorphism $u \xi_{KU}$. Finally, we are ready to prove the main theorem of this paper.
\begin{theo} \label{mainKUtheo} Let $p$ be an odd prime. The functor
$$\R=\R_{KU} \colon \D(\pi_*KU_{(p)}) \to \D(KU_{(p)})$$
together with the natural isomorphism $s \colon \R_{KU} \circ [1] \cong \Sigma \circ \R_{KU}$ is a triangulated equivalence. 

\end{theo}

\begin{proof} We already know from Subsection \ref{introFranke} that the functor $\R_{KU}$ is an equivalence of categories \cite[Corollary 5.2.1]{Pat12}. Using that the functor $\R_{KO}^{(2)}$ together with $a^{(2)}$ is a triangulated equivalence and Corollary \ref{retracttr}, one sees that it suffices to check that the diagram of natural isomorphisms
$$\xymatrix{\R_{KO}^{(2)} \iota^* \circ [1] \ar[rr]^-{\cong} \ar@{=}[d] & & \iota^* \R_{KU} \circ [1] \ar[rr]^-{\iota^*  s}   & &    \iota^* \Sigma \R_{KU} \ar@{=}[d] \\  \R_{KO}^{(2)} \circ [1] \iota^* \ar[rr]^-{a^{(2)} \iota^*}  & & \Sigma \R_{KO}^{(2)} \iota^*  \ar[rr]^{\cong} & & \Sigma \iota^* \R_{KU} }$$
commutes. Here the arrows labeled by the isomorphism sign are the isomorphisms constructed in Proposition \ref{KUKOcom}. Since the identities $s = u \xi_{KU}$ and $a^{(2)} = u \xi_{KO}$ hold, checking commutativity of the latter diagram amounts to checking that the diagram
$$\xymatrix{\R_{KO}^{(2)} \iota^* \circ [1] \ar[rr]^-{\cong} \ar@{=}[d] & & \iota^* \R_{KU} \circ [1] \ar[rr]^-{\iota^* \xi_{KU}}   & &    \iota^* \Sigma \R_{KU} \ar@{=}[d] \\  \R_{KO}^{(2)} \circ [1] \iota^* \ar[rr]^-{\xi_{KO} \iota^*}  & & \Sigma \R_{KO}^{(2)} \iota^*  \ar[rr]^{\cong} & & \Sigma \iota^* \R_{KU} }$$
commutes. This can be seen using the explicit description of the transformations $\xi_{KO}$ and $\xi_{KU}$ and the isomorphism $\R_{KO}^{(2)} \iota^* \cong \iota^* \R_{KU}$. More precisely, having the identities 
$$\iota^* \# = \# \iota^*,$$
and
$$\iota^* [1] =[1] \iota^*$$ 
in mind, we can check that $\theta \odot \xi'_{KU} = \xi'_{KO} \odot \theta$. Here
$$\xymatrix{\xi'_{KU} \colon [1] \circ \Q \ar[r]^-{\cong} & \Q \circ \#}$$
and
$$\xymatrix{\xi'_{KO} \colon [1] \circ \Q \ar[r]^-{\cong} & \Q \circ \#}$$
are the natural isomorphisms defined at the beginning of this subsection. Since forming mates and pasting are compatible (see Subsection \ref{paste}), one obtains the identity $(\xi'_{KU})_{!} \odot (\theta)_{!} = (\theta)_{!} \odot (\xi'_{KO})_{!}$. In other words, the diagram
$$\xymatrix{\Q^{-1} \iota^* \circ [1] \ar[rr]^-{(\theta)_{!} [1]}  \ar@{=}[d] & & \iota^* \Q^{-1} \circ [1] \ar[rr]^-{\iota^* (\xi'_{KU})_{!}} & & \iota^* \# \Q^{-1} \ar@{=}[d]  \\ \Q^{-1} \circ [1] \iota^* \ar[rr]^-{(\xi'_{KO})_{!} \iota^*}  & & \#  \Q^{-1} \iota^* \ar[rr]^-{\# (\theta)_{!}} & &   \#\iota^* \Q^{-1}}$$
commutes. Using $\iota^* \Sigma = \Sigma \iota^*$ and $\iota^* \# = \# \iota^*$, it can be checked similarly that the diagram
$$\xymatrix{\Hocolim \iota^* \# \ar[rr]^-{\kappa \#}  \ar@{=}[d] & & \iota^* \Hocolim \# \ar[rr]^-{\iota^* \xi''_{KU}}  & & \iota^* \Sigma \Hocolim \ar@{=}[d]  \\ \Hocolim \# \iota^* \ar[rr]^-{\xi''_{KO} \iota^*}  & & \Sigma \Hocolim \iota^* \ar[rr]^-{\Sigma \kappa} & & \Sigma \iota^* \Hocolim }$$
commutes. Combining these two diagrams and passing to the derived categories yields the desired commutative diagram which finishes the proof.
\end{proof}




\vspace{0.5cm}

\noindent Department of Mathematical Sciences \\ University of Copenhagen \\
Universitetsparken 5 \\ 2100 Copenhagen $\emptyset$\\ Denmark

\;

\;

\noindent E-mail address: \texttt{irakli.p@math.ku.dk}


\begin{thebibliography}{XXXX}



\bibitem{BKS04}
\textbf{D. Benson, H.Krause, S. Schwede},
{\it Realizability of modules over Tate cohomology},\\
Trans. Amer. Math. Soc., \textbf{356}, No. 9, (2004), 3621-3668.

\bibitem{Ben}
\textbf{R. Bentmann},
{\it Homotopy-theoretic E-theory and n-order},\\
J. Hom. Rel. Struct., \textbf{9}, (2014), 455-463.

\bibitem{B85}
\textbf{A. K. Bousfield},
{\it On the homotopy theory of $K$-local spectra at an odd prime},\\
Amer. J. Math., \textbf{107}, No.4, (1985), 895-932.

\bibitem{B90}
\textbf{A. K. Bousfield},
{\it A classification of K-local spectra},\\
J. Pure Appl. Alg., \textbf{66}, (1990), 121-163.

\bibitem{BJS}
\textbf{U. Bunke, M. Joachim, S. Stolz},
{\it Classifying spaces and spectra representing the K-theory of a graded C*-algebra},\\
High-dimensional manifold topology, World Sci. Publishing, (2003), 80-114.

\bibitem{C98}
\textbf{J. D. Christensen},
{\it Ideals in triangulated categories: Phantoms, ghosts and skeleta},\\
Adv. Math., \textbf{136}, (1998), 284-339.

\bibitem{Cis03}
\textbf{D-C. Cisinski}, 
{\it Images directes cohomologiques dans les cat\'{e}gories de mod\'{e}les},\\ 
Ann. Math. Blaise Pascal, \textbf{10} (2), (2003), 195-244.

\bibitem{CN08}
\textbf{D-C. Cisinski, A Neeman}, 
{\it Additivity for derivator K-theory},\\ 
Adv. Math., \textbf{217}, (2008), 1381-1475.

\bibitem{DEKM}
\textbf{I. Dell'Ambrogio, H. Emerson, T. Kandelaki, R. Meyer},
{\it A functorial equivariant $K$-theory spectrum and an equivariant Lefschetz formula},\\
preprint, http://arxiv.org/abs/1104.3441, (2011).  

\bibitem{DS95}
\textbf{W. G. Dwyer, J. Spalinski},
{\it Homotopy theories and model categories},\\
Handbook of algebraic topology, Elsevier, Amsterdam, (1995), 73-126.


\bibitem{Ehr63}
\textbf{C. Ehresmann},
{\it Cat\'{e}gories structur\'{e}es},\\
Ann. Sci. \'{E}cole Norm. Sup., \textbf{80} (3), (1963), 349-426. 

\bibitem{Ehr65}
\textbf{C. Ehresmann},
{\it Cat\'{e}gories et structures},\\
Dunod, Paris (1965).


\bibitem{F96}
\textbf{J. Franke},
{\it Uniqueness theorems for certain triangulated categories possessing an Adams spectral sequence},\\ preprint, K-theory Preprint Archives 139, http://www.math.uiuc.edu/K-theory/0139/, (1996).


\bibitem{GM96}
\textbf{S. I. Gelfand, Y. I. Manin},
{\it Methods of homological algebra},\\
Springer-Verlag, Berlin, (1996).


\bibitem{GP99}
\textbf{M. Grandis, R. Par\'{e}},
{\it Limits in double categories},\\ 
Cahiers Topol. G\'{e}om. Diff. Cat\'{e}g., \textbf{40}, (1999), 162-220.

\bibitem{G99}
\textbf{J. Greenlees},
{\it Rational $S^1$-equivariant stable homotopy theory},\\
Mem. Amer. Math. Soc., \textbf{138}, No. 661, (1999).

\bibitem{Gro13}
\textbf{M. Groth},
{\it Derivators, pointed derivators and stable derivators},\\
Algebr. Geom. Topol., \textbf{13}, (2013), 313-374.


\bibitem{H97}
\textbf{V. Hinich},
{\it Homological algebra of homotopy algebras},\\
Comm. in Alg., \textbf{25}, No.10, (1997), 3291-3323.



\bibitem{H99}
\textbf{M. Hovey},
{\it Model categories},\\
Mathematical Surveys and Monographs, Amer. Math. Soc., Providence, RI, \textbf{63}, (1999).

\bibitem{HSS00}
\textbf{M. Hovey, J. Smith, B. Shipley},
{\it  Symmetric spectra},\\
J. Amer. Math. Soc., \textbf{13}, (2000), 149-208.

\bibitem{KS74}
\textbf{G. M. Kelly, R. Street}, 
{\it Review of the elements of 2-categories},\\ 
Proceedings of the Sydney Category Theory Seminar, Lecture notes in mathematics \textbf{420}, Springer, Berlin,
(1974), 75-103. 

\bibitem{L03}
\textbf{A. Lazarev},
{\it Towers of $MU$-algebras and the generalized Hopkins-Miller theorem },\\
Proc. Lond. Math. Soc., \textbf{87} (3), (2003), 498-522.

\bibitem{Mah}
\textbf{S. Mahanta},
{\it Colocalizations of noncommutative spectra and bootstrap categories},\\
Adv. Math., \textbf{285}, (2015), 72-100.

\bibitem{Mal07}
\textbf{G. Maltsiniotis},
{\it La $K$-th\'{e}orie d'un d\'{e}rivateur triangul\'{e}},\\ 
Categories in algebra, geometry and mathematical physics,
Contemp. Math., \textbf{431}, Amer. Math. Soc., Providence, RI, (2007), 341-368. 

\bibitem{Mit68}
\textbf{B. Mitchell},
{\it On the dimensions of objects and categories II. Finite ordered sets},\\
J. Alg., \textbf{9}, (1968), 341-368.

\bibitem{MR}
\textbf{F. Muro, G. Raptis},
{\it A note on $K$-theory and triangulated derivators},\\ 
Adv. Math., \textbf{227}, (2011), No. 5, 1827-1845.

\bibitem{Pat12}
\textbf{I. Patchkoria},
{\it On the algebraic classification of module spectra},\\
Algebr. Geom. Topol., \textbf{12}, (2012), 2329-2388.

\bibitem{Q67}
\textbf{D. Quillen},
{\it Homotopical algebra},\\
Lecture notes in mathematics, \textbf{43}, Springer, Berlin, (1967).


\bibitem{Rog08}
\textbf{J. Rognes},
{\it Galois extensions of structured ring spectra. Stably dualizable groups},\\
Mem. Amer. Math. Soc., \textbf{192}, No. 898, (2008).

\bibitem{R07}
\textbf{C. Roitzheim},
{\it  Rigidity and exotic Models for the $K$-local stable homotopy category},\\
Geom. Topol., \textbf{11}, (2007), 1855-1886.

\bibitem{R08}
\textbf{C. Roitzheim},
{\it On the algebraic classification of $K$-local spectra},\\
Homology Homotopy Appl., \textbf{10}, (2008), 389-412.

\bibitem{Schl}
\textbf{M. Schlichting}
{\it A note on $K$-theory and triangulated categories},\\
Invent. Math., \textbf{150}, No. 1, (2002), 111-116.


\bibitem{Schw}
\textbf{S. Schwede},
{\it The stable homotopy category is rigid},\\
Ann. Math., \textbf{166}, (2007), 837-863.

\bibitem{Schw1}
\textbf{S. Schwede},
{\it The $n$-order of algebraic triangulated categories},\\
J. Topol., \textbf{6}, (2013), 857-867.

\bibitem{SS00}
\textbf{S. Schwede, B. Shipley},
{\it Algebras and modules in monoidal model categories},\\
Proc. Lond. Math. Soc., \textbf{80}, (2000), 491-511.

\bibitem{SS03}
\textbf{S. Schwede, B. Shipley},
{\it Stable model categories are categories of modules},\\
Topology, \textbf{42}, (2003), 103-153.

\bibitem{S04}
\textbf{B. Shipley},
{\it A convenient model category for commutative ring spectra},\\
Homotopy theory: relations with algebraic geometry, group cohomology, and algebraic $K$-theory, Contemp. Math., \textbf{346}, Amer. Math. Soc., Providence, RI, (2004), 473-483.

\bibitem{S07}
\textbf{B. Shipley},
{\it  $H\Z$-algebra spectra are differential graded algebras},\\
Amer. J. Math., \textbf{129}, (2007), 351-379.

\bibitem{W94}
\textbf{C. A. Weibel}, 
{\it An introduction to homological algebra},\\ Cambridge Studies in Advanced Mathematics. 38., 
Cambridge Univ. Press, Cambridge, (1994).

\bibitem{W98}
\textbf{J. J. Wolbert},
{\it  Classifying modules over $K$-theory spectra},\\
J. Pure Appl. Alg., \textbf{124}, (1998), 289-323.





\end{thebibliography}
\end{document}